\documentclass[11pt,reqno]{amsart}

\usepackage[a4paper,margin=1in]{geometry}
\usepackage{amsmath,amssymb,amsthm,mathtools,mathrsfs,bbm}
\usepackage{hyperref}

\usepackage{tikz}
\usetikzlibrary{calc,angles,positioning}

\newtheorem{theorem}{Theorem}[section]
\newtheorem{proposition}[theorem]{Proposition}

\newtheorem{lemma}[theorem]{Lemma}
\newtheorem{question}[theorem]{Question}

\newtheorem*{claim}{Claim}

\numberwithin{equation}{section}

\newcommand{\norm}[1]{\lVert #1\rVert_\infty}
\newcommand{\supp}{\operatorname{supp}}
\newcommand{\dd}{\,\mathrm{d}}

\renewcommand{\leq}{\leqslant}
\renewcommand{\le}{\leqslant}
\renewcommand{\geq}{\geqslant}
\renewcommand{\ge}{\geqslant}

\title[Rectangles, triangles and Schr\"{o}dinger waves]{Rectangles, triangles and Schr\"{o}dinger waves}

\author[Bennett]{Jonathan Bennett}
\address[Jonathan Bennett]{School of Mathematics, The Watson Building, University of Birmingham, Edgbaston, Birmingham, B15 2TT, England.}
\email{J.Bennett@bham.ac.uk}

\author[Kova\v{c}]{Vjekoslav Kova\v{c}}
\address[Vjekoslav Kova\v{c}]{Department of Mathematics, Faculty of Science, University of Zagreb, Bijeni\v{c}ka cesta 30, 10000 Zagreb, Croatia.}
\email{vjekovac@math.hr}

\author[Nakamura]{Shohei Nakamura}
\address[Shohei Nakamura]{School of Mathematics, The Watson Building, University of Birmingham, Edgbaston, Birmingham, B15 2TT, England.}
\email{s.nakamura@bham.ac.uk}

\author[Oliveira]{Itamar Oliveira}
\address[Itamar Oliveira]{School of Mathematics, The Watson Building, University of Birmingham, Edgbaston, Birmingham, B15 2TT, England.}
\email{i.oliveira@bham.ac.uk, oliveira.itamar.w@gmail.com}

\subjclass[2020]{42B10, 44A12, 05B40}
\keywords{Schr\"{o}dinger equation, weighted estimate, Mizohata--Takeuchi estimate, Wigner transform, lattice combinatorics}

\begin{document}

\begin{abstract}
Can a finite set of lattice points determine many rectangles and few isosceles triangles? This turns out to be a surprisingly interesting question in combinatorial geometry that we answer using basic analytic number theory combined with a finite-field construction. The result is useful because it gives obstructions to Mizohata--Takeuchi-type estimates in the setting of the paraboloid. Specifically, we establish transference between Euclidean and periodic weighted $\mathrm{L}^2$ estimates for solutions to the Schr\"{o}dinger equation, and then relate the failure of the latter to quantities tied to combinatorial problems, such as the one above. By completing this programme we give new explicit combinatorial counterexamples to the paraboloid case of the Mizohata--Takeuchi conjecture, which was recently shown to be false by Cairo for curved hypersurfaces.
\end{abstract}

\maketitle

\tableofcontents

%%%%%

\section{Introduction}

Numerous problems in combinatorial geometry concern rectangles or triangles determined by points in the integer lattice $\mathbb{Z}^2$. We begin with a simple-looking question of this type involving both rectangles and triangles, and then go on to reveal a surprising connection with the behaviour of solutions to the Schr\"{o}dinger equation. Readers with a combinatorial inclination may also find the question of independent interest.

%Such estimates already have well-established links with distance-set questions, although usually from the viewpoint of geometric measure theory rather than combinatorial geometry; see Mattila \cite{Matt} and the more recent work of Du, Guth, Ou, Wang, Wilson and Zhang \cite{DuGuthOuWangWilsonZhang}. 
Configuration counts already have well-established links with problems in harmonic analysis and dispersive PDE, some of the most prominent examples being: 
counts of lattice solutions arising from moments of exponential sums \cite{B1},
counts of polynomial Diophantine relations associated with periodic KdV moments \cite{B2},
counts of repeated angles, parallelograms or rectangles to prove extension estimates for finite-field paraboloids \cite{RS18,IKL20,Lew20}
and counts of parallelograms in frequency sets to prove Strichartz estimates for periodic Schr\"{o}dinger equations \cite{HK,LZ25}. However, the appearance of triangles from such problems in harmonic analysis seems to be much more unusual.
We will need to show that rectangles inside a subset of the grid can greatly outnumber isosceles triangles, and such (direct) comparative counts of geometric configurations also appear to be relatively uncommon in the literature. Certain comparative counts are common in additive combinatorics, where counts of arithmetic progressions (and other configurations) in a set are controlled by the Gowers uniformity norms, which themselves count parallelotopes; see \cite{TV, THOFA}. One may also
compare counts of parallelotopes of different dimensions using the log-convexity of Gowers norms \cite{Shk14,Man17,BT24}.

\subsection{A problem in combinatorial geometry}
For a finite set $A\subseteq\mathbb{Z}^2$, let $\Box(A)$ be the set of \emph{ordered rectangles} in $A$:
\[ \Box(A) := \{(k_1,k_2,k_3,k_4)\in A^4 \,:\, k_1+k_3=k_2+k_4,\ (k_2-k_1)\cdot(k_4-k_1)=0 \}, \]
where $\cdot$ denotes the usual dot product in $\mathbb{Z}^2\subseteq\mathbb{R}^2$; see Figure~\ref{fig:rect_and_trian}.
Also, let $\triangle(A)$ be the set of \emph{ordered isosceles triangles} in $A$:
\[ \triangle(A) := \{(k_1,k_2,k_3)\in A^3 \,:\, (2k_3-k_1-k_2)\cdot(k_2-k_1)=0\}. \]
For $k\in A$, let $\triangle^k(A)$ denote the collection of ordered isosceles triangles in $A$ whose distinguished apex-vertex is $k$:
\[ \triangle^k(A) := \{(k_1,k_2,k)\in A^3 \,:\, (2k-k_1-k_2)\cdot(k_2-k_1)=0\}, \]
so that
\[ \triangle(A) = \bigcup_{k\in A} \triangle^k(A),\quad \text{i.e.,}\quad \#\triangle(A) = \sum_{k\in A} \#\triangle^k(A); \]
see Figure~\ref{fig:rect_and_trian} again.
Degenerate configurations are included throughout; this is important for the problem, since degeneracies can contribute significantly to the total count. For instance, $(k_1,k_1,k_1,k_1)$ and $(k_1,k_1,k_3,k_3)$ are examples of degenerate rectangles, while $(k_1,k_1,k_1)$ and $(k_1,k_1,k_3)$ are degenerate isosceles triangles.

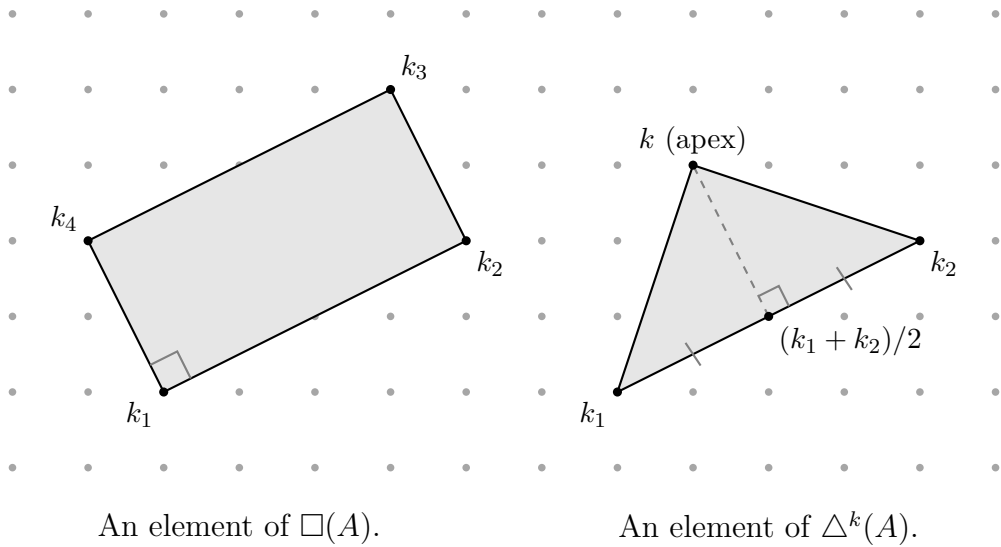
\begin{figure}
\begin{center}
\begin{tikzpicture}[>=stealth]
    \foreach \x in {0,...,13} {
        \foreach \y in {0,...,6} {
            \fill[gray!70] (\x,\y) circle (0.05);
        }
    }
    \coordinate (m1) at (2,1);
    \coordinate (m2) at (6,3);
    \coordinate (m3) at (5,5);
    \coordinate (m4) at (1,3);
    \draw[thick, black, fill=gray!20] (m1) -- (m2) -- (m3) -- (m4) -- cycle;
    \pic [draw, gray, thick, angle radius=4mm] {right angle = m2--m1--m4};
    \fill[black] (m1) circle (0.06) node[below left] {$k_1$};
    \fill[black] (m2) circle (0.06) node[below right] {$k_2$};
    \fill[black] (m3) circle (0.06) node[above right] {$k_3$};
    \fill[black] (m4) circle (0.06) node[above left] {$k_4$};
    \node[black, font=\large] at (3, -0.8) {An element of $\Box(A)$.};
    \coordinate (t1) at (8,1);
    \coordinate (t2) at (12,3);
    \coordinate (a)  at (9,4);
    \coordinate (M)  at ($(t1)!0.5!(t2)$);
    \draw[thick, black, fill=gray!20] (t1) -- (t2) -- (a) -- cycle;
    \draw[thick, dashed, gray] (M) -- (a);
    \pic [draw, gray, thick, angle radius=3mm] {right angle = a--M--t2};
    \fill[black] (t1) circle (0.06) node[below left] {$k_1$};
    \fill[black] (t2) circle (0.06) node[below right] {$k_2$};
    \fill[black] (a)  circle (0.06) node[above] {$k$ (apex)};
    \fill[black] (M)  circle (0.06) node[below right] {$(k_1+k_2)/2$};
    \draw[thick, gray] ($(t1)!0.5!(M) + (-0.1, 0.15)$) -- ($(t1)!0.5!(M) + (0.1, -0.15)$);
    \draw[thick, gray] ($(M)!0.5!(t2) + (-0.1, 0.15)$) -- ($(M)!0.5!(t2) + (0.1, -0.15)$);
    \node[black, font=\large] at (10, -0.8) {An element of $\triangle^k(A)$.};
\end{tikzpicture}
\end{center}
\caption{Rectangles and isosceles triangles on the integer lattice.}
\label{fig:rect_and_trian}
\end{figure}

We ask whether the number of rectangles in a finite set $A\subseteq\mathbb{Z}^2$ is necessarily controlled by the number of isosceles triangles in $A$.

\begin{question}\label{quest:combinatorics}
Is there a constant $M>0$ such that
\begin{equation}\label{boxtriangle}
\#\Box(A) \leq M \#\triangle(A)
\end{equation}
holds for every finite set $A\subseteq\mathbb{Z}^2$? A similar question can be asked about the weaker ``pinned'' inequality
\begin{equation}\label{pinnedboxtriangle}
\#\Box(A) \leq M \,\#A\, \max_{k\in A}\#\triangle^k(A).
\end{equation}
\end{question}

Inequalities \eqref{boxtriangle} and \eqref{pinnedboxtriangle} should be compared with the estimate
\begin{equation}\label{pachsharir}
\#\Box(A)\lesssim (\#A)^2\log\#A
\end{equation}
of Pach and Sharir \cite{PachSharir}, which is best possible if the right-hand side depends only on the cardinality of $A$. Pach and Sharir also considered bounds for $\#\triangle(A)$, proving that
\begin{equation}\label{pachiso}
\#\triangle(A)\lesssim (\#A)^{7/3},
\end{equation}
and this has since been improved by Pach and Tardos \cite{PachTardos}. Both \eqref{pachsharir} and \eqref{pachiso} hold for arbitrary finite subsets of the plane, not just subsets of the integer lattice. For context, upper bounds on $\#\triangle(A)$ quickly imply distance-set estimates (see \cite{PachSharir,PachTardos, TV}), suggesting that $\#\triangle(A)$ may also satisfy a near-quadratic bound similar to \eqref{pachsharir}.

Both parts of Question~\ref{quest:combinatorics} have negative answers. We quantify their failure in two natural finitary ways. Define
\[ \mathfrak{R}(N) := \max_{\emptyset\ne A\subseteq\mathbb{Z}^2\cap[-N,N]^2} \frac{\#\Box(A)}{\#\triangle(A)},
\quad \mathfrak{R}^{\mathrm{pin}}(N) := \max_{\emptyset\ne A\subseteq\mathbb{Z}^2\cap[-N,N]^2} \frac{\#\Box(A)}{\#A\, \max_{k\in A}\limits \#\triangle^k(A)}, \]
and
\[ \widetilde{\mathfrak{R}}(n) := \sup_{\substack{A\subseteq\mathbb{Z}^2\\ \#A=n}} \frac{\#\Box(A)}{\#\triangle(A)},
\quad \widetilde{\mathfrak{R}}^{\mathrm{pin}}(n) := \sup_{\substack{A\subseteq\mathbb{Z}^2\\ \#A=n}} \frac{\#\Box(A)}{\#A\, \max_{k\in A}\limits \#\triangle^k(A)}. \]
In words, $\mathfrak{R}(N)$ and $\mathfrak{R}^{\mathrm{pin}}(N)$ measure the blowup of the best constants in \eqref{boxtriangle} and \eqref{pinnedboxtriangle} as $A$ ranges over subsets of a bounded part of the integer lattice $\mathbb{Z}^2\cap[-N,N]^2$, while $\widetilde{\mathfrak{R}}(n)$ and $\widetilde{\mathfrak{R}}^{\mathrm{pin}}(n)$ measure the blowup as $A$ varies over all $n$-element subsets of the lattice.

Here is our main combinatorial result.

\begin{theorem}\label{thm:combinatorics}
There is a constant $c>0$ such that, for all sufficiently large $N$,
\[ \mathfrak{R}(N) \geq \mathfrak{R}^{\mathrm{pin}}(N) \geq (\log N)^{c/\log\log\log N}, \]
and, for all sufficiently large positive integers $n$,
\[ \widetilde{\mathfrak{R}}(n) \geq \widetilde{\mathfrak{R}}^{\mathrm{pin}}(n) \geq (\log n)^{c/\log\log\log n}. \]
\end{theorem}

The lower estimates from Theorem~\ref{thm:combinatorics} are close to optimal because they are complemented by easy upper bounds. Namely, combining \eqref{pachsharir} with the trivial observation $\#\triangle(A)\geq(\#A)^2$, we obtain
\[ \widetilde{\mathfrak{R}}^{\mathrm{pin}}(n) \leq \widetilde{\mathfrak{R}}(n) \lesssim \log n. \]
Applying the same to a subset $A\subseteq\mathbb{Z}^2\cap[-N,N]^2$, which necessarily has cardinality $\#A\lesssim N^2$, gives
\[ \mathfrak{R}^{\mathrm{pin}}(N) \leq \mathfrak{R}(N) \lesssim \log N. \]

The construction of the sets that achieve the bounds stated in Theorem~\ref{thm:combinatorics} was found by OpenAI's ChatGPT 5.5 Pro; for details see the section declaring the AI usage. The authors ran it numerous times shortly after the announcement of the breakthrough by an internal OpenAI model on the Erd\H{o}s unit distance conjecture \cite{OpenAIErdos,UnitErdos}. Even though the two problems are very different, there is a certain similarity between the proofs, both in the local-to-global transference from simpler fields to the reals, and in the use of certain sieving when counting occurrences of (un)desired patterns. This vague analogy is even more apparent in relation to subsequent papers \cite{SumProduct,SplitPrimes,CMPSY} enabled by the aforementioned breakthrough, which appeared after our initial discovery, while we were writing up the results. Related local-to-global arguments may be found elsewhere in harmonic analysis, with \cite{HW18} being a natural example to mention.

\subsection{Strichartz estimates for the Schr\"{o}dinger equation}
Now we turn to the analytic side of the paper.

The classical \emph{Strichartz estimates} for a solution $u\colon\mathbb{R}^d\times\mathbb{R}\rightarrow\mathbb{C}$ of the \emph{free Schr\"{o}dinger equation}
\begin{equation*}%\label{SchrodingerEuc}
2\pi i\frac{\partial}{\partial t} u(x,t) = \Delta_x u(x,t); \quad (x,t)\in\mathbb{R}^d\times \mathbb{R},
\end{equation*}
with initial data $u_0\in \mathrm{L}^2(\mathbb{R}^d)$ are space-time norm inequalities of the form
\begin{equation}\label{cts}
\|u(x,t)\|_{\mathrm{L}^q_t\mathrm{L}^p_x(\mathbb{R}^d\times\mathbb{R})}\lesssim\|u_0\|_{\mathrm{L}^2(\mathbb{R}^d)}.
\end{equation}
The time-integrability captured by such estimates is a reflection of the dispersive nature of the Schr\"{o}dinger equation. These inequalities and their variants have had many applications, notably in the well-posedness theory of nonlinear Schr\"{o}dinger equations; see for example \cite{TaoDisp}. 
The Lebesgue exponents $p,q$ for which \eqref{cts} holds have been known for some time, beginning with the original work of Strichartz \cite{Stric}, and culminating in \cite{KT}. From the point of view of harmonic analysis, these are oscillatory integral estimates arising from the validity of the explicit formula
\begin{equation}\label{eq:Eucl-solution} 
u(x,t)=\int_{\mathbb{R}^d}\widehat{u_0}(\xi) e^{2\pi i(x\cdot \xi+t|\xi|^2)}\dd \xi
\end{equation}
for suitable initial data $u_0$.

It is also natural to consider Strichartz inequalities in periodic settings. For a function $u_0\in \mathrm{L}^2(\mathbb{T}^d)$, where $\mathbb{T}^d$ denotes the $d$-dimensional torus, the Strichartz estimates for the \textit{periodic Schr\"{o}dinger equation}
\begin{equation*}%\label{Schrodingerper}
2\pi i\frac{\partial}{\partial t} u(x,t) = \Delta_x u(x,t); \quad (x,t)\in\mathbb{T}^d\times \mathbb{T},
\end{equation*}
take a similar form
\begin{equation}\label{disc}
\|u(x,t)\|_{\mathrm{L}^q_t\mathrm{L}^p_x(\mathbb{T}^d\times\mathbb{T})}\lesssim\|u_0\|_{\mathrm{L}^2(\mathbb{T}^d)}.
\end{equation}
However, these are now exponential sum estimates, with
the solution $u$ given by the formula
\begin{equation}\label{eq:per-solution}
u(x,t)=\sum_{k\in\mathbb{Z}^d}\widehat{u_0}(k)e^{2\pi i(x\cdot k+t|k|^2)},
\quad (x,t)\in\mathbb{T}^d\times\mathbb{T}
\end{equation}
for suitable initial data $u_0$.
The family of inequalities \eqref{disc} appears to be considerably more subtle than its Euclidean counterpart \eqref{cts}, and the full range of admissible exponents $p,q$ remains to be understood.
Much is known, however, beginning with \cite{B1} and culminating recently in \cite{BD} in the pure-norm ($p=q$) case; see \cite{Dem,Nah}.
Characterisation of the mixed-norm ($p\neq q$) case is still open. In this direction, we mention \cite{BGT}, in which they established \eqref{disc} for a wide range of $p,q$ by allowing a certain loss of regularity. In \cite{KNS}, motivated by a problem in nonlinear dispersive PDEs, the authors proposed the question of the validity of \eqref{disc} at the Keel--Tao endpoint with an $\varepsilon$-loss of regularity; as far as we know, this remains open as well.
Related to this problem, it was recently observed in \cite{DJLM} that the corresponding mixed-norm decoupling estimates turn out to be false. 
As may therefore be expected, the admissible exponents in \eqref{cts} and \eqref{disc} are well known to be quite different, reflecting common distinctions between oscillatory integrals and exponential sums, the latter relating to aspects of combinatorial geometry and number theory. In particular, it is not in general possible to pass between the continuous and discrete inequalities \eqref{cts} and \eqref{disc} for given exponents $p,q$. 

\subsection{Weighted Strichartz estimates on \texorpdfstring{$\mathbb{R}^{d+1}$}{Euclidean space}}
In weighted $\mathrm{L}^2$ settings, however, such distinctions may be less apparent. 
There is a considerable literature on weighted Strichartz estimates of various forms, often involving specific classes of weights; see, for example, \cite{BRV} and the references there. In certain situations, weighted Strichartz estimates are referred to as \textit{Morawetz estimates}, having their origins in \cite{Mor}; see \cite{TaoDisp}. 
Here we are interested in $\mathrm{L}^2$ quantities of the form
\begin{equation}\label{weightednorm}
\int_{\mathbb{R}} \int_{\mathbb{R}^d} |u(x,t)|^2w(x,t)\dd x\dd t,
\end{equation}
where $w$ is an arbitrary weight, by which we mean a nonnegative locally integrable function. There is a growing literature on $\mathrm{L}^2$ estimates for such expressions, partly due to their relation with problems in geometric measure theory; see, for example, \cite{Du} and the references there.
%The weight indicates which parts of space-time are emphasised. %The natural endpoint control would come from integrating the weight along the straight rays followed by Schr\"{o}dinger wave packets.

One way to generate conjectural estimates for the general weighted norm \eqref{weightednorm} is to first reformulate it as a \textit{phase-space integral}, as recently described in \cite{DMV, BGNO}. The starting point is the \emph{Wigner transform} $W\colon \mathrm{L}^2(\mathbb{R}^d)\times \mathrm{L}^2(\mathbb{R}^d)\rightarrow \mathrm{L}^\infty(\mathbb{R}^d\times\mathbb{R}^d)$ given by
\begin{equation}\label{eucwig}
W(g_1,g_2)(x,v):=\int_{\mathbb{R}^d}g_1\Bigl(x+\frac{y}{2}\Bigr)\overline{g_2\Bigl(x-\frac{y}{2}\Bigr)} e^{-2\pi i v\cdot y}\dd y; \quad (x,v)\in\mathbb{R}^d\times\mathbb{R}^d,
\end{equation} 
and the well-known observation \cite{Wigner} that the function 
\[ f(x,v,t) := W\bigl(u(\cdot,t), u(\cdot,t)\bigr)(x,v); \quad (x,v,t)\in\mathbb{R}^d\times\mathbb{R}^d\times\mathbb{R} \]
satisfies the \emph{kinetic transport equation} 
\begin{equation}\label{kineuc}
\frac{\partial}{\partial t}f(x,v,t) = 2v\cdot\nabla_x f(x,v,t),
\end{equation}
meaning that 
\[ W(u(\cdot,t), u(\cdot,t))(x,v)=W(u_0,u_0)(x+2tv,v). \]
It then follows by an application of the classical spatial marginal property
\[ \int_{\mathbb{R}^d}W(g,g)(x,v)\dd v=|g(x)|^2 \]
that
\begin{equation}\label{PSrepcont}
|u(x,t)|^2=\rho W(u_0,u_0)(x,t),
\end{equation}
where the operator $\rho$, acting on a suitable phase-space function $f_0$, is given by
\[ \rho f_0(x,t):=\int_{\mathbb{R}^d}f_0(x+2tv,v)\dd v. \]
By duality this leads to the desired phase-space integral representation
\[ \int_{\mathbb{R}}\int_{\mathbb{R}^d}|u(x,t)|^2w(x,t)\dd x\dd t=\int_{\mathbb{R}^d}\int_{\mathbb{R}^d}W(u_0,u_0)(x,v)\rho^{\ast}w(x,v)\dd x\dd v, \]
where we identify
the operator 
\[ \rho^{\ast}w(x,v)=\int_{\mathbb{R}}w(x-2tv,t)\dd t \]
as a certain \emph{parametrised space-time X-ray transform}: for each $x,v\in\mathbb{R}^d$ it computes an integral of the weight $w$ along the space-time line through $(x,0)$ in the direction $(-2v,1)$.
Finally, noting the frequency marginal property 
\[ \int_{\mathbb{R}^d}W(g,g)(x,v)\dd x=|\widehat{g}(v)|^2, \]
one is then tentatively led to the following conjectures.

\begin{claim}[{Parabolic Stein-type estimate, motivated by \cite{SteinProb}}]
For each $d\geq1$ there is a constant $C$ such that
\begin{equation}\label{eq:euc-Stein}
\int_0^1\int_{\mathbb{R}^d}|u(x,t)|^2w(x,t)\dd x\dd t
\le C\int_{\mathbb{R}^d}|\widehat{u_0}(v)|^2\sup_{x\in\mathbb{R}^d}\rho^{\ast}w(x,v)\dd v
\end{equation}
for all nonnegative weights $w$ and all $u_0$ with compact Fourier support.
\end{claim}

\begin{claim}[{Parabolic Mizohata--Takeuchi estimate, motivated by \cite{Tak74,Tak80,Miz85,BRV}}]
For each $d\geq1$ there is a constant $C$ such that
\begin{equation}\label{eq:euc-MT}
\int_0^1\int_{\mathbb{R}^d}|u(x,t)|^2w(x,t)\dd x\dd t
\le C\sup_{\substack{x\in\mathbb{R}^d\\ v\in\supp\widehat{u_0}}}\!\!\!\!\rho^{\ast}w(x,v)\ \|u_0\|_{\mathrm{L}^2(\mathbb{R}^d)}^2
\end{equation}
for all $w$ and $u_0$ as before. 
\end{claim}
The reader may have noticed that the time localisations on the left-hand sides of \eqref{eq:euc-Stein} and \eqref{eq:euc-MT} do not feature in our Wigner-function heuristics. However, an elementary scaling and limiting argument reveals that \eqref{eq:euc-Stein} and \eqref{eq:euc-MT} are equivalent to the fully global and scale-invariant inequalities
\begin{equation}\label{eq:euc-Stein-glo}
\int_{\mathbb{R}}\int_{\mathbb{R}^d}|u(x,t)|^2w(x,t)\dd x\dd t
\le C\int_{\mathbb{R}^d}|\widehat{u_0}(v)|^2\sup_{x\in\mathbb{R}^d}\rho^{\ast}w(x,v)\dd v,
\end{equation}
and 
\begin{equation}\label{eq:euc-MT-glo}
\int_{\mathbb{R}}\int_{\mathbb{R}^d}|u(x,t)|^2w(x,t)\dd x\dd t
\le C\sup_{\substack{x\in\mathbb{R}^d\\ v\in\supp\widehat{u_0}}}\!\!\!\!\rho^{\ast}w(x,v)\ \|u_0\|_{\mathrm{L}^2(\mathbb{R}^d)}^2
\end{equation}
respectively.
Indeed, if one wishes one may also spatially localise the left-hand sides of \eqref{eq:euc-Stein} and \eqref{eq:euc-MT}, replacing $[0,1]\times\mathbb{R}^d$ with $[0,1]^{d+1}$, without penalty; one simply applies a parabolic scaling to such an estimate to deduce that
\[ \int_{[0,M]^{d}\times [0,M^2]}|u(x,t)|^2w(x,t)\dd x\dd t\leq C\int_{\mathbb{R}^d}|\widehat{u_0}(v)|^2\sup_{x\in\mathbb{R}^d}\rho^{\ast}w_M(x,Mv)\dd v, \]
uniformly in $M\geq 1$, where $w_M(x,t)=M^2w(Mx,M^2t)$, and notices that $\rho^{\ast}w_M(x,Mv)=\rho^{\ast}w(Mx,v)$.
Alternatively, it is straightforward to observe that one may harmlessly impose a frequency localisation in the global inequalities \eqref{eq:euc-Stein-glo} and \eqref{eq:euc-MT-glo}, restricting attention to initial data $u_0$ with Fourier support in the unit ball.

Clearly, \eqref{eq:euc-MT} is weaker than \eqref{eq:euc-Stein}.
These classical conjectures have often been formulated by replacing the underlying paraboloid $\{(\xi,|\xi|^2):\xi\in\mathbb{R}^d\}$ either with a more general hypersurface, or simply with the sphere. We refer to Sections 1 and 2 of \cite{BGNO} for some discussion on the origins of these problems and how the specific inequalities \eqref{eq:euc-Stein} and \eqref{eq:euc-MT} fit into this wider geometric framework. We note that one may pass back and forth between these parabolic forms and the more familiar spherical forms using parabolic scaling and tensoring arguments; see \cite{Ferrante}.

Very recently, these two claims were shown to be false by Cairo, who constructed a counterexample to the Mizohata--Takeuchi conjecture for every non-flat $\mathrm{C}^2$ hypersurface \cite{Cairo}. (We note that there is nonetheless some evidence in their favour; see for example \cite{Mul,BGNO} and the references there.)
It is then natural to seek to quantify the failure of these estimates. 
For $N\ge1$, define $\mathfrak{S}_{\mathbb{R}^d}(N)$ to be the least constant $C$ in \eqref{eq:euc-Stein} under the additional hypothesis
\[ \supp\widehat{u_0}\subseteq[-N,N]^d, \]
and define $\mathfrak{M}_{\mathbb{R}^d}(N)$ similarly using \eqref{eq:euc-MT}. Thus, the fact that estimates \eqref{eq:euc-Stein} and \eqref{eq:euc-MT} do not hold can be restated as 
\[ \sup_N\mathfrak{S}_{\mathbb{R}^d}(N)=\infty \quad\text{and}\quad \sup_N\mathfrak{M}_{\mathbb{R}^d}(N)=\infty. \]
Cairo's counterexample \cite{Cairo} actually gives $\mathfrak{S}_{\mathbb{R}^d}(N) \geq \mathfrak{M}_{\mathbb{R}^d}(N) \gtrsim \log N$, while Cairo and Zhang \cite{CairoZhang} (also see \cite{Fenves}) later obtained counterexamples with a power loss, $N^{d/(d+k)-\varepsilon}$, for certain families of $\mathrm{C}^k$ convex hypersurfaces in $\mathbb{R}^{d+1}$; we refer to \cite{Carbery09} for earlier indications of the relevance of combinatorial geometry to Mizohata--Takeuchi-type inequalities.
Complementary upper bounds of order $N^{d/(d+2)+\varepsilon}$ for all strictly convex $\mathrm{C}^2$ hypersurfaces were provided earlier by the work of Carbery, Iliopoulou and Wang \cite{CIW}; see also Guth \cite{Guthtalk}. 

\subsection{Weighted Strichartz estimates on \texorpdfstring{$\mathbb{T}^{d+1}$}{torus}}%\label{Sect:introStrich}
The phase-space point of view developed in \cite{BGNO} remains relevant, and recent exponential-sum work such as \cite{Yang} suggests that periodic and arithmetic models may deserve separate attention.

In order to formulate our periodic conjectures, we mimic the heuristics from the previous section, beginning with a Wigner transform on the torus,
\begin{equation*}%\label{perwig}
W_{\mathbb{T}^{d}}(g_1,g_2)(x,\kappa):=\int_{\mathbb{T}^d}g_1(x+y)\overline{g_2(x-y)} e^{-2\pi i \kappa\cdot y}\dd y; \quad (x,\kappa)\in\mathbb{T}^d\times\mathbb{Z}^d.
\end{equation*}
We note that our normalisation of $W_{\mathbb{T}^d}$ is slightly different from \eqref{eucwig}, since the factors $1/2$ are missing from the input functions.
We do this ``oversampling'' in order to avoid an inconvenient aliasing phenomenon (see \cite{CM}) that would otherwise arise in the frequency marginal, interfering with the desired periodic analogue of the phase-space representation \eqref{PSrepcont}. 
%This does, however, have the effect of interfering with the frequency marginal, as we will see in a moment.

As in the Euclidean setting, for a periodic Schr\"{o}dinger solution
\eqref{eq:per-solution}, the function
\[ f(x,\kappa,t) := W_{\mathbb{T}^d}\bigl(u(\cdot,t),u(\cdot,t)\bigr)(x,\kappa);\quad (x,\kappa,t)\in\mathbb{T}^d\times\mathbb{Z}^d\times\mathbb{T} \]
satisfies a kinetic transport equation: this time the discrete-velocity equation
\begin{equation*}%\label{kte}
\frac{\partial}{\partial t}f(x,\kappa,t) = \kappa\cdot\nabla_x f(x,\kappa,t).
\end{equation*}
Again, we note a difference between this equation and its Euclidean counterpart \eqref{kineuc}, arising from the slight difference in the definitions of the Euclidean and periodic Wigner distributions. 
From the differential equation, or directly, we obtain the following identity: if $u$ is the periodic Schr\"{o}dinger solution \eqref{eq:per-solution}, then
\[ W_{\mathbb{T}^d}(u(\cdot,t),u(\cdot,t))(x,\kappa)
=W_{\mathbb{T}^d}(u_0,u_0)(x+t\kappa,\kappa). \]
Using the marginal property
\[ \sum_{\kappa\in\mathbb{Z}^d}W_{\mathbb{T}^d}(g,g)(x,\kappa)=|g(x)|^2 \]
and Fubini's theorem,
this leads to the phase-space representation
\begin{equation}\label{discpsrep}
|u(x,t)|^2=\widetilde{\rho}_{\mathbb{Z}^{d}}W_{\mathbb{T}^d}(u_0,u_0)(x,t),
\end{equation}
where $\widetilde{\rho}_{\mathbb{Z}^d}$ is the discrete phase-space operator
\[ \widetilde{\rho}_{\mathbb{Z}^d}f_0(x,t):=\sum_{\kappa\in\mathbb{Z}^d}f_0(x+t\kappa,\kappa). \]
Observe that the adjoint of $\widetilde{\rho}_{\mathbb{Z}^d}$, defined on $1$-periodic space-time functions $w$ by
\begin{equation*}%\label{eq:per-xray}
\widetilde{\rho}^{\ast}_{\mathbb{Z}^d}w(x,\kappa)=\int_0^1w(x-t\kappa,t)\dd t,
\end{equation*}
is now a (normalised) geodesic X-ray transform on the space-time torus $\mathbb{T}^{d+1}$ with integer velocities.
In particular, for a weight function $w\geq0$ on $\mathbb{T}^d\times\mathbb{T}$, by \eqref{discpsrep} we have the phase-space integral representation
\[ \int_{\mathbb{T}} \int_{\mathbb{T}^d} |u(x,t)|^2w(x,t)\dd x\dd t
= \sum_{\kappa\in\mathbb{Z}^d}\int_{\mathbb{T}^d} W_{\mathbb{T}^d}(u_0,u_0)(x,\kappa)\widetilde{\rho}^{\ast}_{\mathbb{Z}^d}w(x,\kappa)\dd x. \]
This representation is the periodic analogue of the Euclidean phase-space identity mentioned in the previous subsection and used in \cite{BGNO}.
Noting the other marginal property,
\[ \int_{\mathbb{T}^d}W_{\mathbb{T}^d}(g,g)(x,\kappa)\dd x
= \begin{cases}
|\widehat{g}(\kappa/2)|^2 & \text{if } \kappa\in 2\mathbb{Z}^d, \\ 
0 & \text{if } \kappa\not\in 2\mathbb{Z}^d,
\end{cases} \]
one might tentatively formulate \emph{periodic} (or \emph{discrete}) \emph{Stein} and \emph{Mizohata--Takeuchi conjectures} as follows.
Denote
\[ \rho^{\ast}_{\mathbb{Z}^d}w(x,k) := \widetilde{\rho}^{\ast}_{\mathbb{Z}^d}w(x,2k) = \int_0^1 w(x-2tk,t)\dd t; \quad (x,k)\in\mathbb{T}^d\times\mathbb{Z}^d. \]

\begin{claim}[Periodic Stein-type estimate]%\label{Conj:PStein}
For each $d\geq 1$ there exists a constant $C$ such that
\begin{equation}\label{eq:per-Stein}
\int_{\mathbb{T}} \int_{\mathbb{T}^d} \left|u(x,t)\right|^2w(x,t)\dd x\dd t\leq C\sum_{k\in\mathbb{Z}^d}|\widehat{u_0}(k)|^2\sup_{x\in\mathbb{T}^d}\rho^{\ast}_{\mathbb{Z}^{d}}w(x,k),
\end{equation}
for all weight functions $w\geq0$ and suitable $u_0\in \mathrm{L}^2(\mathbb{T}^d)$. 
\end{claim}

\begin{claim}[Periodic Mizohata--Takeuchi estimate]%\label{Conj:PMT}
For each $d\geq 1$ there exists a constant $C$ such that
\begin{equation}\label{eq:per-MT}
\int_{\mathbb{T}} \int_{\mathbb{T}^d} \left|u(x,t)\right|^2w(x,t)\dd x\dd t\leq C \sup_{\substack{x\in\mathbb{T}^d\\ k\in\supp\widehat{u_0}}}\rho^{\ast}_{\mathbb{Z}^{d}}w(x,k) \,\|u_0\|_{\mathrm{L}^2(\mathbb{T}^d)}^2,
\end{equation}
for $w$ and $u_0$ as before. 
\end{claim}

It will turn out that these two estimates do not hold either, just as their Euclidean counterparts do not -- a fact that we establish with Theorems~\ref{thm:config-lower-intro} and \ref{thm:d1-log-intro} below; see also Section~\ref{Sect:pertoEuc}. It is therefore interesting to quantify their failure as well.
%For a finite set of frequencies $A\subseteq\mathbb{Z}^d$ and a tuple of complex coefficients $a=(a_k)_{k\in A}$, write
%\[ \mathrm{E}_A a(x,t):=\sum_{k\in A} a_k e^{2\pi i(x\cdot k+t|k|^2)}; \quad (x,t)\in\mathbb{T}^d\times\mathbb{T}. \]
For $N\ge1$, let $\mathfrak{S}_{\mathbb{T}^d}(N)$ and $\mathfrak{M}_{\mathbb{T}^d}(N)$ respectively be the least constants $C$ such that the Stein-type estimate \eqref{eq:per-Stein} and the weaker Mizohata--Takeuchi-type estimate \eqref{eq:per-MT} hold for all nonnegative weights $w$ and all trigonometric polynomials $u_0\in\mathrm{L}^2(\mathbb{T}^d)$ such that
\[ \supp\widehat{u_0} \subseteq \mathbb{Z}^d \cap [-N,N]^d. \]
Similar numerical quantities can be defined by fixing the cardinality of the frequency support of $u_0$, which is a specific feature of the periodic setting. Namely, for a positive integer $n$, let $\widetilde{\mathfrak{S}}_{\mathbb{T}^d}(n)$ and $\widetilde{\mathfrak{M}}_{\mathbb{T}^d}(n)$ respectively be the least constants $C$ such that \eqref{eq:per-Stein} and \eqref{eq:per-MT} hold for weights $w\geq0$ and trigonometric polynomials $u_0\in\mathrm{L}^2(\mathbb{T}^d)$ such that
\[ \#\supp\widehat{u_0} = n. \]
Clearly,
\begin{equation}\label{eq:per-MT_vs_St}
\mathfrak{M}_{\mathbb{T}^d}(N) \leq \mathfrak{S}_{\mathbb{T}^d}(N) \quad\text{and}\quad \widetilde{\mathfrak{M}}_{\mathbb{T}^d}(n) \leq \widetilde{\mathfrak{S}}_{\mathbb{T}^d}(n).
\end{equation}

The following quantitative transference result from the local Euclidean model to the periodic model relates this subsection to the previous one.

\begin{theorem}\label{thm:quant-transference-intro}
For every $d\ge1$ and $N\ge1$ we have
\begin{equation}\label{eq:St-transference}
\mathfrak{S}_{\mathbb{T}^d}(N)\le \mathfrak{S}_{\mathbb{R}^d}(N+1),
\end{equation}
and
\begin{equation}\label{eq:MT-transference}
\mathfrak{M}_{\mathbb{T}^d}(N)\le \mathfrak{M}_{\mathbb{R}^d}(N+1).
\end{equation}
\end{theorem}

Consequently, every lower bound in the periodic Stein-type and Mizohata--Takeuchi estimates is also a ``local'' Euclidean lower bound.
In Section~\ref{sec:transference} we will show directly that \eqref{eq:euc-Stein} and \eqref{eq:per-Stein} (and also \eqref{eq:euc-MT} and \eqref{eq:per-MT}) are mutually equivalent, even though we now know that they are both false. However, the reverse implication will only be qualitative, unlike Theorem~\ref{thm:quant-transference-intro}.

\subsection{Consequences of the combinatorial problem}
We now explain how the rectangle--triangle comparison influences the local constants. The following theorem translates the configuration problem from the beginning of the introduction into lower bounds for these constants.

\begin{theorem}\label{thm:config-lower-intro}
In dimension $d=2$, for every $N\geq1$ and every positive integer $n$, we have
\begin{equation}\label{eq:Stein-config-lower-intro}
\mathfrak{S}_{\mathbb{T}^2}(N)\geq\mathfrak{R}(N), \quad \widetilde{\mathfrak{S}}_{\mathbb{T}^2}(n)\geq\widetilde{\mathfrak{R}}(n)
\end{equation}
and
\begin{equation}\label{eq:MT-config-lower-intro}
\mathfrak{M}_{\mathbb{T}^2}(N)\geq\mathfrak{R}^{\mathrm{pin}}(N), \quad \widetilde{\mathfrak{M}}_{\mathbb{T}^2}(n)\geq\widetilde{\mathfrak{R}}^{\mathrm{pin}}(n).
\end{equation}
\end{theorem}

Combining Theorems~\ref{thm:combinatorics},~\ref{thm:quant-transference-intro} and~\ref{thm:config-lower-intro} gives the following lower bounds in dimension $d=2$: there is a constant $c>0$ such that, for all sufficiently large $N$,
\begin{equation}\label{eq:logtologloglog}
\mathfrak{S}_{\mathbb{T}^2}(N),\ \mathfrak{M}_{\mathbb{T}^2}(N),\ \mathfrak{S}_{\mathbb{R}^2}(N),\ \mathfrak{M}_{\mathbb{R}^2}(N) \ge (\log N)^{c/\log\log\log N}.
\end{equation}
These bounds are weaker than the logarithmic lower bound by Cairo mentioned earlier and revisited\footnote{Here we are glossing over a technical difference arising from the nature of the space-time truncations used in formulating the local problems; here our localisations are temporal, whereas in \cite{Cairo} they are space-time.} below, but they have a different feature: the examples on $\mathbb{T}^2\times\mathbb{T}$ are witnessed by explicit weights of the form $w=|u|^2$.
From Theorem~\ref{thm:config-lower-intro} one also gets nontrivial lower bounds on $\widetilde{\mathfrak{S}}_{\mathbb{T}^2}(n)$ and $\widetilde{\mathfrak{M}}_{\mathbb{T}^2}(n)$, but these will be significantly improved in \eqref{eq:justlinear} below.

The strongest lower bounds in the paper actually come from different trigonometric weights.
These are more in the spirit of the counterexample by Cairo \cite{Cairo}, although they have the virtue that they are of a purely combinatorial nature.
We discovered this construction by modifying the previously described approach, again with the help of ChatGPT 5.5 Pro.

\begin{theorem}\label{thm:d1-log-intro}
For every $d\ge1$ and all sufficiently large $N$,
\begin{equation}\label{eq:justlog1}
\mathfrak{S}_{\mathbb{T}^d}(N),\ \mathfrak{M}_{\mathbb{T}^d}(N)\gtrsim_d (\log N)^d.
\end{equation}
Also, for every $d\ge1$ and all sufficiently large positive integers $n$ we have
\begin{equation}\label{eq:justlinear}
\widetilde{\mathfrak{S}}_{\mathbb{T}^d}(n),\ \widetilde{\mathfrak{M}}_{\mathbb{T}^d}(n)\sim_d n.
\end{equation}
\end{theorem}

The proof of Theorem~\ref{thm:d1-log-intro} will use nonnegative lacunary trigonometric polynomial weights, rather than weights of the form $|u|^2$. 

As an immediate consequence of Theorems~\ref{thm:quant-transference-intro} and~\ref{thm:d1-log-intro}, for every $d\ge1$ and all sufficiently large $N$,
\begin{equation}\label{eq:justlog2}
\mathfrak{S}_{\mathbb{R}^d}(N),\ \mathfrak{M}_{\mathbb{R}^d}(N)\gtrsim_d (\log N)^d.
\end{equation}
In words, both the sharp periodic and the sharp Euclidean constants have at least polylogarithmic growth in every dimension. As we have said, the $\log N$ lower bound has already been observed by Cairo, and here we improve it slightly to $(\log N)^d$, largely because the paraboloid is more symmetric than a general surface.
Conversely, the only known upper bounds on $\mathfrak{M}_{\mathbb{R}^d}(N)$ are still of power type in $N$; see \cite{CIW}.

The quantities $\widetilde{\mathfrak{S}}_{\mathbb{T}^d}(n)$ and $\widetilde{\mathfrak{M}}_{\mathbb{T}^d}(n)$ are less subtle, and we observe that Theorem~\ref{thm:d1-log-intro} determines their growth very precisely. These are, however, specific to the periodic setting (where we have discrete frequencies) and do not have analogues in the Euclidean model.

\subsection{Summary, notation and paper organisation}
In light of the existing counterexamples to the Stein-type and Mizohata--Takeuchi conjectures \eqref{eq:euc-Stein} and \eqref{eq:euc-MT}, the contributions of this paper are:
\begin{itemize}
\item the identification of the periodic analogues of these conjectures;
\item a transference principle for lower bounds from the periodic setting to the Euclidean one;
\item a reduction of lower bounds for best constants to solvable geometric problems; and
\item solutions of these problems with tools from combinatorics and number theory.
\end{itemize}
Unlike the general counterexamples of Cairo \cite{Cairo}, Cairo and Zhang \cite{CairoZhang}, and Fenves \cite{Fenves}, the rectangle--triangle construction is not designed to give universal examples or examples with power loss, but rather aims to explain the failure of \eqref{eq:euc-Stein} and \eqref{eq:euc-MT} in an elementary way.

For nonnegative quantities $X$ and $Y$, the estimate $X\lesssim Y$ means that $X\leq CY$ for a constant $C$ independent of the variables being quantified in the relevant statement. Similarly, $X\gtrsim Y$ means $Y\lesssim X$, and $X\sim Y$ means that both $X\lesssim Y$ and $Y\lesssim X$ hold. If a subscript is attached, as in $\lesssim_d$, $\gtrsim_d$ or $\sim_d$, then the implicit constant is allowed to depend on the subscripted parameters, and on no other varying quantities. 

The notation $X=O(Y)$, or simply $O(Y)$ inside an expression, means that a real quantity $X$ is bounded in absolute value by $CY$ for an implicit constant $C$. A subscript has the same meaning as above: for example, $O_{\alpha,\beta,\gamma}(Y)$ denotes an error whose absolute value is at most $C_{\alpha,\beta,\gamma}Y$, where the constant may depend on $\alpha,\beta,\gamma$ but is independent of the other variables under consideration.

We use $\#E$ for the cardinality of a finite set $E$, and $\supp f$ for the support of a function or sequence $f$. We write $\mathbb{T}^d := \mathbb{R}^d/\mathbb{Z}^d$ for the $d$-dimensional torus, $\mathbb{F}_p$ for the field of $p$ elements, and $\mathbb{Z}_q := \mathbb{Z}/q\mathbb{Z}$ for the cyclic group of order $q\geq1$. The notation $\norm{x}$ denotes the $\ell^\infty$ norm of a vector $x\in\mathbb{R}^d$.

The rest of the paper is organised as follows. Section~\ref{sec:rectangle-triangle} proves Theorem~\ref{thm:combinatorics} by constructing finite lattice sets with many rectangles and comparatively few isosceles triangles. Section~\ref{sec:transference} proves the quantitative Euclidean-to-periodic transference Theorem~\ref{thm:quant-transference-intro} and its purely qualitative converse. Section~\ref{sec:config-lower} translates the rectangle--triangle counts into lower bounds for the periodic constants $\mathfrak{S}_{\mathbb{T}^2}(N)$ and $\mathfrak{M}_{\mathbb{T}^2}(N)$ in the form of Theorem~\ref{thm:config-lower-intro}. Finally, Section~\ref{sec:d1-log-weight} alters the canonical choice of weight $|u|^2$ to further improve the lower bounds on $\mathfrak{S}_{\mathbb{T}^d}(N)$ and $\mathfrak{M}_{\mathbb{T}^d}(N)$ in any ambient dimension $d$ and establishes Theorem~\ref{thm:d1-log-intro}.

%%%%

\section{The rectangle--triangle construction}
\label{sec:rectangle-triangle}

In this section we answer Question~\ref{quest:combinatorics} in the negative and prove Theorem~\ref{thm:combinatorics}.

For a prime $p\equiv3\pmod4$, consider the finite-field ``unit circle'' defined as
\begin{equation*}%\label{eq:Bp_circle}
B_p:=\{(\alpha,\beta)\in\mathbb{F}_p^2 \,:\, \alpha^2+\beta^2=1\};
\end{equation*}
see Figure~\ref{fig:unit_circle}.
If $P$ is a nonempty finite set of primes congruent to $3$ modulo $4$, put $q=\prod_{p\in P}p$.
Using the Chinese remainder theorem, the group $\mathbb{Z}_q$ can be identified with the finite group product $\prod_{p\in P}\mathbb{F}_p$. Under the same identification, define $B_q\subseteq\mathbb{Z}_q^2$ by
\[ x\in B_q \quad\Longleftrightarrow\quad
x\bmod p\in B_p\text{ for every }p\in P. \]
The lattice set used in the proof is then defined as
\[ A_N(q) := \{k\in\mathbb{Z}^2\cap[-N,N]^2 \,:\, k\bmod q\in B_q\}, \]
or, more explicitly,
\[ A_N(q) = \{ (\alpha,\beta)\in\mathbb{Z}^2\cap[-N,N]^2 \,:\, \alpha^2+\beta^2 \equiv 1 \!\!\!\pmod q \}. \]
A special case is illustrated in Figure~\ref{fig:rec_tri_set_A}. We will, of course, need to consider arbitrarily large $q$ and $N$.

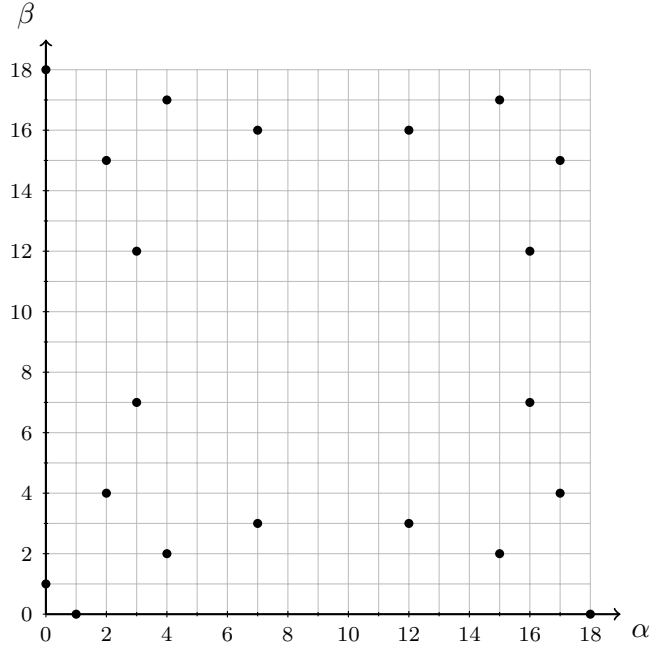
\begin{figure}
\begin{center}
\begin{tikzpicture}[scale=0.4]
    \draw[step=1cm, gray!50, thin] (0,0) grid (18,18);
    \draw[thick, ->] (0,0) -- (19,0) node[anchor=north west] {$\alpha$};
    \draw[thick, ->] (0,0) -- (0,19) node[anchor=south east] {$\beta$};
    \foreach \i in {0, 2, 4, 6, 8, 10, 12, 14, 16, 18} {
        \draw (\i, 3pt) -- (\i, -3pt) node[anchor=north, font=\scriptsize] {\i};
        \draw (3pt, \i) -- (-3pt, \i) node[anchor=east, font=\scriptsize] {\i};
    }
    \foreach \i in {1, 3, 5, 7, 9, 11, 13, 15, 17} {
        \draw (\i, 2pt) -- (\i, -2pt);
        \draw (2pt, \i) -- (-2pt, \i);
    }
    \foreach \x/\y in {
        0/1, 0/18, 
        1/0, 18/0, 
        2/4, 2/15, 17/4, 17/15, 
        4/2, 4/17, 15/2, 15/17, 
        3/7, 3/12, 16/7, 16/12, 
        7/3, 7/16, 12/3, 12/16%
    } {
        \fill[black] (\x, \y) circle (0.15);
    }
\end{tikzpicture}
\end{center}
\caption{The unit circle $B_{19}$ in $\mathbb{F}_{19}^2$.}
\label{fig:unit_circle}
\end{figure}

\begin{figure}
\begin{center}
\hspace*{7mm}\begin{tikzpicture}[scale=0.145]
    \def\N{50} % Bounding box [-N,N]^2
    \def\q{21} % Product of primes p = 3 (mod 4)
    \draw[step=1, gray!40, thin] (-\N.5, -\N.5) grid (\N.5, \N.5);
    \draw[->, thick, darkgray] (-\N.5, 0) -- (\N.5+0.5, 0) node[right] {$x$};
    \draw[->, thick, darkgray] (0, -\N.5) -- (0, \N.5+0.5) node[above] {$y$};
    \draw[thick, dashed, gray] (-\N, -\N) rectangle (\N, \N);
    \foreach \x in {-\N,...,\N} {
        \foreach \y in {-\N,...,\N} {
            \pgfmathparse{int(mod((\x)*(\x) + (\y)*(\y), \q))}
            \ifnum\pgfmathresult=1
                \fill[black] (\x, \y) circle (0.3);
            \fi
        }
    }
\end{tikzpicture}
\end{center}
\caption{The set $A_{50}(21)$. Here we have $P=\{3,7\}$.}
\label{fig:rec_tri_set_A}
\end{figure}
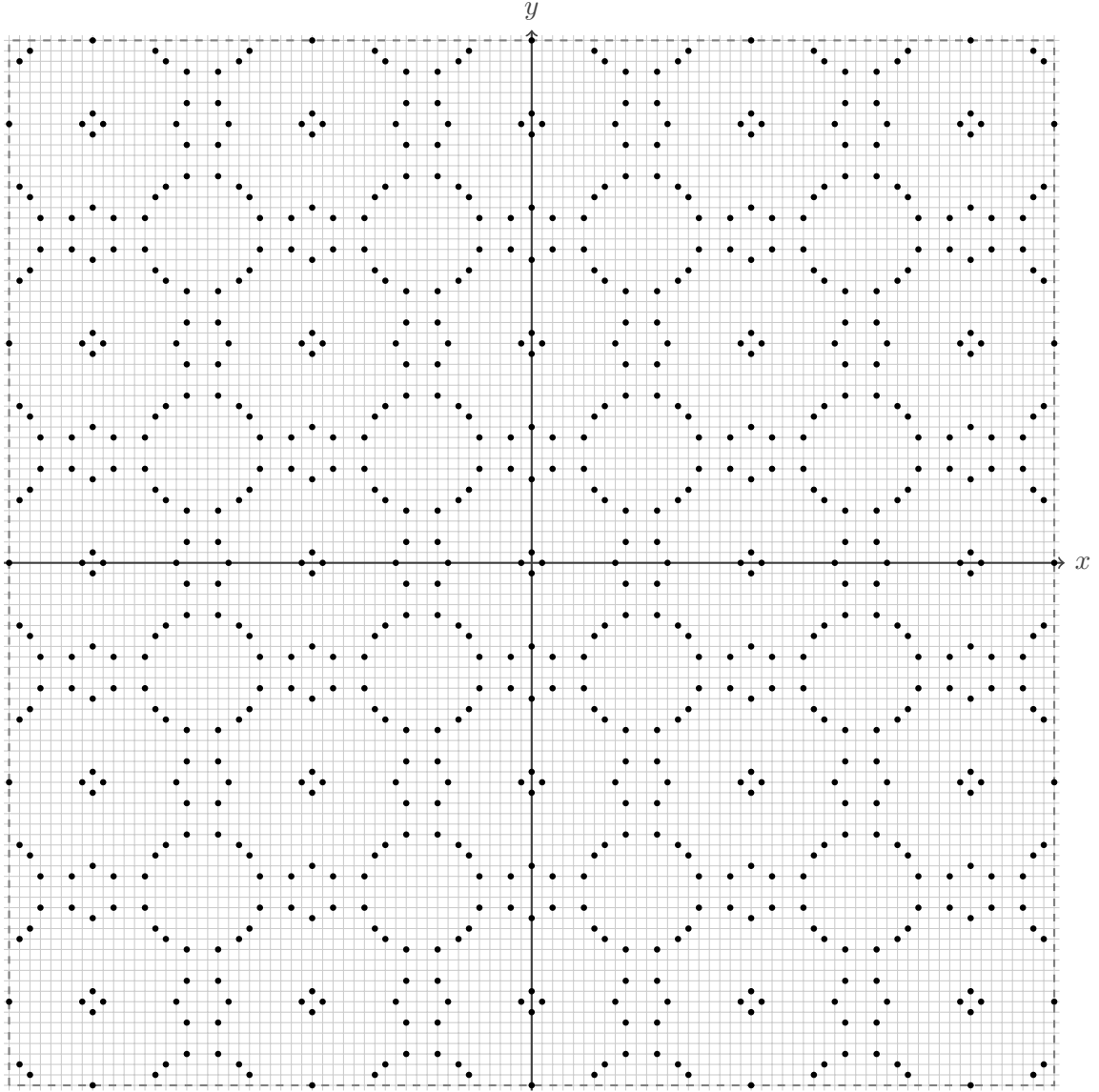

For a nonempty finite $A\subseteq\mathbb{Z}^2$, write
\[ \mathfrak{Q}(A):=\frac{\#\Box(A)}{\#\triangle(A)},
\quad \mathfrak{Q}^{\mathrm{pin}}(A):=
\frac{\#\Box(A)}{\#A\,\max_{k\in A}\limits \#\triangle^k(A)}. \]

The following proposition is a more concrete variant of Theorem~\ref{thm:combinatorics}.

\begin{proposition}\label{prop:fixedmodulus}
There exist absolute constants $c_0$ and $C_0$ such that the following holds. Let $P\neq\emptyset$ be a finite set of primes congruent to $3$ modulo $4$. Define $q=\prod_{p\in P}p$. If
\begin{equation}\label{eq:fixeducondition}
N\geq e^{C_0 q^2},
\end{equation}
then
\begin{equation}\label{eq:fixeduniform}
\mathfrak{Q}(A_N(q)),\ \mathfrak{Q}^{\mathrm{pin}}(A_N(q))
\geq c_0 \Bigl(\frac{4}{3}\Bigr)^{\#P}.
\end{equation}
\end{proposition}

Note that Proposition~\ref{prop:fixedmodulus} already gives a negative answer to Question~\ref{quest:combinatorics}, simply by taking arbitrarily large sets $P$ of primes congruent to $3$ modulo $4$. This is certainly possible by Dirichlet's theorem on primes in arithmetic progressions. Both quantities in \eqref{eq:fixeduniform} become arbitrarily large as soon as $N$ is sufficiently large that \eqref{eq:fixeducondition} holds.

The rest of the section proves Proposition~\ref{prop:fixedmodulus} and then derives Theorem~\ref{thm:combinatorics}.

\subsection{Finite-field and cyclic-group constructions}
Write $J(\alpha,\beta):=(-\beta,\alpha)$. We use the same symbol on $\mathbb{Z}^2$, $\mathbb{F}_p^2$ and $\mathbb{Z}_q^2$. Likewise, $\cdot$ denotes the standard dot product on $\mathbb{Z}^2$ and the analogous bi-additive form on the finite rings.
It will be useful to rewrite $B_p$ as
\[ B_p = \{x\in\mathbb{F}_p^2 \,:\, x\cdot x=1\}. \]

For $S\subseteq\mathbb{F}_p^2$, where $p$ is an odd prime, define
\begin{align*}
\Box_{\mathbb{F}_p}(S):=\{(x,w,r,s) \,:\ & x\in\mathbb{F}_p^2,
\ w\in \mathbb{F}_p^2\setminus\{(0,0)\},\ r,s\in\mathbb{F}_p,\\
&x,\ x+rw,\ x+sJw,\ x+rw+sJw\in S\},\\
\triangle_{\mathbb{F}_p}(S):=\{(x,w,r,s) \,:\ & x\in\mathbb{F}_p^2,
\ w\in \mathbb{F}_p^2\setminus\{(0,0)\},\ r,s\in\mathbb{F}_p,\\
&x,\ x+rw,\ x+(rw+sJw)/2\in S\}.
\end{align*}
Here $1/2$ denotes the inverse of $2$ in $\mathbb{F}_p$.
The parametrisation clearly overcounts: different choices of $w,r,s$ can give the same geometric configuration. These parametrised counts are the ones that lift cleanly to the lattice $\mathbb{Z}^2$.
For $k\in \mathbb{F}_p^2$, we also define the pinned parametrised triangle set
\begin{align*}
\triangle^k_{\mathbb{F}_p}(S) := \{(w,r,s) \,:\ & w\in\mathbb{F}_p^2\setminus\{(0,0)\}, \ r,s\in\mathbb{F}_p,\\
& k-(rw+sJw)/2\in S,
\ k+(rw-sJw)/2\in S\}.
\end{align*}
The two displayed points are the base points of a triangle whose apex is fixed at $k$.

%From basic facts on quadratic residues we know that the set \eqref{eq:Bp_circle} has $\#B_p=p+1$ points. 
We need to count the elements of $B_p$, $\Box_{\mathbb{F}_p}(B_p)$, $\triangle_{\mathbb{F}_p}(B_p)$ and $\triangle^k_{\mathbb{F}_p}(B_p)$.

\begin{lemma}\label{lm:field}
For every prime $p\equiv3\pmod4$ we have
\begin{align} 
\#B_p & = p+1, \label{eq:Fpcount1} \\
\#\Box_{\mathbb{F}_p}(B_p) & = 4(p-1)p(p+1), \label{eq:Fpcount2} \\
\#\triangle_{\mathbb{F}_p}(B_p) & = 3(p-1)p(p+1), \label{eq:Fpcount3}
\end{align}
and, for every $k\in B_p$,
\begin{equation}
\#\triangle^k_{\mathbb{F}_p}(B_p) = 3(p-1)p. \label{eq:Fpcount4}
\end{equation}
\end{lemma}

\begin{proof}
We first count the points of $B_p$. Since $p\equiv3\pmod4$, the element $-1$ is not a square in $\mathbb{F}_p$ (i.e., it is a quadratic non-residue), so $1+t^2\ne0$ for every $t\in\mathbb{F}_p$. The $p$ points
\[ (\alpha,\beta) = \left(\frac{1-t^2}{1+t^2},\frac{2t}{1+t^2}\right); \quad t\in\mathbb{F}_p, \]
lie in $B_p$, and so does $(-1,0)$. For each of these parametrised points $(\alpha,\beta)$, one has $1+\alpha=2/(1+t^2)$ and hence $t=\beta/(1+\alpha)$, so they are distinct and none of them is $(-1,0)$. Conversely, if $(\alpha,\beta)\in B_p$ and $\alpha\ne-1$, then, with $t=\beta/(1+\alpha)$,
\[ t^2=\frac{1-\alpha}{1+\alpha},\quad \alpha=\frac{1-t^2}{1+t^2},\quad \beta=\frac{2t}{1+t^2}. \]
The remaining case $\alpha=-1$ forces $\beta=0$. This verifies \eqref{eq:Fpcount1}.

Let us count $\#\Box_{\mathbb{F}_p}(B_p)$. Denote $e=(1,0)$. For $x\in B_p$, define an $\mathbb{F}_p$-linear operator $\Phi_x\colon\mathbb{F}_p^2\to\mathbb{F}_p^2$ by
\[ \Phi_x(y):=(y\cdot x,\,y\cdot Jx). \]
The vectors $x$ and $Jx$ satisfy
\[ x\cdot x=Jx\cdot Jx=1, \quad x\cdot Jx=0. \]
Thus, $\Phi_x$ is invertible, preserves the dot product, maps $x$ to $e$ and commutes with $J$. Consequently, for each fixed $x\in B_p$, we may replace $x$ by $e$ and $w$ by $z=\Phi_x(w)$; as $w$ runs through $\mathbb{F}_p^2\setminus\{(0,0)\}$, so does $z$.
Write $z=(a,b)\neq(0,0)$. Since $p\equiv3\pmod4$, we have $z\cdot z=a^2+b^2\ne0$. For $r,s\in\mathbb{F}_p$,
\[ e+rz\in B_p \quad\Longleftrightarrow\quad r\bigl(2a+r(z\cdot z)\bigr)=0 \]
and
\[ e+sJz\in B_p \quad\Longleftrightarrow\quad s\bigl(-2b+s(z\cdot z)\bigr)=0. \]
Thus, the first condition has $1+\mathbbm{1}_{a\ne0}$ solutions for $r$, while the second has $1+\mathbbm{1}_{b\ne0}$ solutions for $s$. Moreover, because $z\cdot Jz=0$, the fourth vertex condition follows from these two side conditions:
\begin{align*} 
& (e+rz+sJz)\cdot(e+rz+sJz)-1 \\
& = \bigl((e+rz)\cdot(e+rz)-1\bigr) + \bigl((e+sJz)\cdot(e+sJz)-1\bigr) = 0,
\end{align*}
so $e+rz+sJz\in B_p$.
For fixed $x$, the number of rectangle parameters is therefore
\[ \sum_{z=(a,b)\neq(0,0)} (1+\mathbbm{1}_{a\ne0})(1+\mathbbm{1}_{b\ne0}) =4(p-1)^2+4(p-1)=4(p-1)p. \]
Multiplying by $\#B_p=p+1$ gives \eqref{eq:Fpcount2}.

Let us recall an elementary identity from number theory before proceeding. Let $\chi$ be the quadratic character on $\mathbb{F}_p$, extended by $\chi(0)=0$. Thus, $\chi(x)=1$ when $x$ is a nonzero quadratic residue and $\chi(x)=-1$ when $x$ is a non-residue. Then it holds that
\begin{equation}\label{eq:equals_minus1} 
\sum_{t\in\mathbb{F}_p}\chi(1+t^2)=-1. 
\end{equation}
Indeed, $\sum_{t\in\mathbb{F}_p}(1+\chi(1+t^2))$ counts the pairs $(u,t)\in\mathbb{F}_p^2$ satisfying $u^2-t^2=1$. Since $p$ is odd, the change of variables $(u,t)\mapsto(u-t,u+t)=(\alpha,\beta)$ is bijective, and the equation becomes $\alpha \beta=1$, which clearly has $p-1$ solutions. This proves \eqref{eq:equals_minus1}.

We now count $\#\triangle_{\mathbb{F}_p}(B_p)$. For fixed $x$ and $z=(a,b)\neq(0,0)$, the two conditions are
\[ e+rz\in B_p, \quad e+\frac{rz+sJz}{2}\in B_p. \]
The first condition is again $r(2a+r(z\cdot z))=0$, while the second is
\begin{equation}\label{eq:circletri}
(z\cdot z)(r^2+s^2)+4ar-4bs=0.
\end{equation}
If $r=0$, then \eqref{eq:circletri} becomes
\[ s\bigl((z\cdot z)s-4b\bigr)=0, \]
which gives $1+\mathbbm{1}_{b\ne0}$ choices for $s$. If $r\neq 0$, then $a\neq 0$ and the first condition has the additional root $r=-2a/(z\cdot z)$; substituting it into \eqref{eq:circletri} gives
\[ (z\cdot z)s^2-4bs-\frac{4a^2}{z\cdot z}=0. \]
This quadratic equation has discriminant $16(a^2+b^2)=16(z\cdot z)$, and hence it has $1+\chi(z\cdot z)$ solutions. Therefore, for each fixed $x$, the number of triangle parameters is
\[ \sum_{z=(a,b)\neq(0,0)} \left(1+\mathbbm{1}_{b\ne0} + \mathbbm{1}_{a\ne0}(1+\chi(a^2+b^2))\right). \]
The first two terms contribute $(p^2-1)+p(p-1)$. For the remaining term,
\[ \sum_{\substack{a\ne0\\ b\in\mathbb{F}_p}} \bigl(1+\chi(a^2+b^2)\bigr)
= p(p-1)+(p-1)\sum_{t\in\mathbb{F}_p}\chi(1+t^2) = (p-1)^2, \]
where we used \eqref{eq:equals_minus1} in the last equality.
Thus the total count for fixed $x$ is
\[ (p^2-1)+p(p-1)+(p-1)^2=3(p-1)p, \]
and multiplying by $\#B_p=p+1$ gives \eqref{eq:Fpcount3}.

It remains to compute the pinned counts. For $u=(u_1,u_2)\in B_p$ define an $\mathbb{F}_p$-linear operator $\Psi_u\colon\mathbb{F}_p^2\to\mathbb{F}_p^2$,
\[ \Psi_u(v_1,v_2):=(u_1v_1-u_2v_2,\,u_2v_1+u_1v_2). \]
The columns of $\Psi_u$ are $u$ and $Ju$, so $\Psi_u$ preserves the dot product, preserves $B_p$, commutes with $J$ and maps $e$ to $u$. Hence $(w,r,s)\mapsto(\Psi_u w,r,s)$ is a bijection from $\triangle^e_{\mathbb{F}_p}(B_p)$ to $\triangle^u_{\mathbb{F}_p}(B_p)$. The cardinality $\#\triangle^k_{\mathbb{F}_p}(B_p)$ is therefore independent of $k\in B_p$.
Every tuple $(x,w,r,s)\in\triangle_{\mathbb{F}_p}(B_p)$ has apex
\[ k=x+\frac{rw+sJw}{2}\in B_p. \]
Conversely, once this apex $k$ is fixed, the two base-point conditions are exactly those in the definition of $\triangle^k_{\mathbb{F}_p}(B_p)$. Hence
\[ \#\triangle_{\mathbb{F}_p}(B_p) = \sum_{k\in B_p}\#\triangle^k_{\mathbb{F}_p}(B_p). \]
Since the summands are equal and $\#B_p=p+1$, each of them is $3(p-1)p$, which proves \eqref{eq:Fpcount4}.
\end{proof}

For a square-free odd positive integer $q$, write
\[ W_q:=\{w\in\mathbb{Z}_q^2:
w\bmod p\ne(0,0)\text{ for each prime }p\mid q\}. \]
For $S\subseteq\mathbb{Z}_q^2$, define
\begin{align*}
\Box_{\mathbb{Z}_q}(S):=\{(x,w,r,s) \,:\ &x\in\mathbb{Z}_q^2,
\ w\in W_q,
\ r,s\in\mathbb{Z}_q,\\
&x,\ x+rw,\ x+sJw,\ x+rw+sJw\in S\},\\
\triangle_{\mathbb{Z}_q}(S):=\{(x,w,r,s) \,:\ &x\in\mathbb{Z}_q^2,
\ w\in W_q,
\ r,s\in\mathbb{Z}_q,\\
&x,\ x+rw,\ x+(rw+sJw)/2\in S\}.
\end{align*}
For $k\in\mathbb{Z}_q^2$ we also define
\begin{align*}
\triangle^k_{\mathbb{Z}_q}(S) := \{(w,r,s) \,:\ &w\in W_q,
\ r,s\in\mathbb{Z}_q,\\
& k-(rw+sJw)/2\in S, \ k+(rw-sJw)/2\in S\}.
\end{align*}
Here $1/2$ is, again, the inverse of $2$ modulo $q$.

\begin{lemma}\label{lm:crt}
Let $P$ be a nonempty finite set of primes congruent to $3$ modulo $4$. Put $q=\prod_{p\in P}p$. Then
\begin{align*}
\#B_q & = \prod_{p\in P} (p+1), \\
\#\Box_{\mathbb{Z}_q}(B_q) & = \prod_{p\in P} 4(p-1)p(p+1), \\
\#\triangle_{\mathbb{Z}_q}(B_q) & = \prod_{p\in P} 3(p-1)p(p+1),
\end{align*}
and, for every $k\in B_q$,
\[ \#\triangle^k_{\mathbb{Z}_q}(B_q) = \prod_{p\in P} 3(p-1)p. \]
Consequently,
\[ \frac{\#\Box_{\mathbb{Z}_q}(B_q)}{\#\triangle_{\mathbb{Z}_q}(B_q)} = \Bigl(\frac{4}{3}\Bigr)^{\#P},
\quad \frac{\#\Box_{\mathbb{Z}_q}(B_q)}{\#B_q\,\max_{k\in B_q}\#\triangle^k_{\mathbb{Z}_q}(B_q)} = \Bigl(\frac{4}{3}\Bigr)^{\#P}. \]
\end{lemma}

\begin{proof}
Since the primes in $P$ are pairwise coprime, the Chinese
remainder theorem gives a ring isomorphism
\[ \pi\colon\mathbb{Z}_q\to \prod_{p\in P}\mathbb{F}_p,
\quad \alpha\mapsto (\alpha\!\!\!\!\mod p)_{p\in P}. \]
This isomorphism also maps the inverse of $2$ in $\mathbb{Z}_q$ to the tuple of inverses of $2$ in each $\mathbb{F}_p$.
We also extend $\pi$ to the Cartesian powers of $\mathbb{Z}_q$ as
\begin{align*} 
& \pi\colon\mathbb{Z}_q^m\to \prod_{p\in P}\mathbb{F}_p^m \cong \Bigl(\prod_{p\in P}\mathbb{F}_p\Bigr)^m, \\
& \pi(\alpha_1,\ldots,\alpha_m) := \bigl((\alpha_1,\ldots,\alpha_m)\!\!\!\!\mod p\bigr)_{p\in P} \equiv \big(\pi(\alpha_1),\ldots,\pi(\alpha_m)\big)
\end{align*}
for every positive integer $m$.
These maps are still bijective. In particular, $\pi$ maps $W_q$ onto
\[ \pi(W_q) = \prod_{p\in P}\bigl(\mathbb{F}_p^2\setminus\{(0,0)\}\bigr), \]
preserves the dot product,
\[ \pi(x\cdot y) = \pi(x) \cdot \pi(y); \quad x,y\in\mathbb{Z}_q^2, \]
and commutes with $J$,
\[ \pi(Jx) = J\pi(x); \quad x\in\mathbb{Z}_q^2. \]
By our definitions,
\begin{align*}
\pi(B_q) & = \prod_{p\in P} B_p, \\
\pi\bigl(\Box_{\mathbb{Z}_q}(B_q)\bigr) & = \prod_{p\in P}\Box_{\mathbb{F}_p}(B_p), \\
\pi\bigl(\triangle_{\mathbb{Z}_q}(B_q)\bigr) & = \prod_{p\in P}\triangle_{\mathbb{F}_p}(B_p),
\end{align*}
and
\[ \pi\bigl(\triangle^k_{\mathbb{Z}_q}(B_q)\bigr) = \prod_{p\in P}\triangle^{k_p}_{\mathbb{F}_p}(B_p) \]
for every $k\in B_q$ such that $\pi(k)=(k_p)_{p\in P}$.
Computing the cardinalities of these four Cartesian products using Lemma~\ref{lm:field} immediately proves all four claimed identities.
\end{proof}

\subsection{Lattice lifting}
%\label{sec:lifting}

A vector $w=(w_1,w_2)\in\mathbb{Z}^2$ is called \emph{primitive} if $w_1$ and $w_2$ are coprime. The following lemma converts modular ``orthogonality'' into exact lattice orthogonality. We state both the unpinned and pinned counts, and keep the dependence on $q$ explicit, so that the later proofs can refer to a single uniform counting result.

\begin{lemma}\label{lm:lifting}
Let $q$ be an odd square-free positive integer and put
\[ \eta_q:=\frac8{q^2}\prod_{p\nmid q}\Bigl(1-\frac1{p^2}\Bigr)>0, \]
where the product is over primes $p$ not dividing $q$.

\emph{(i)} Let
$E\subseteq\mathbb{Z}_q^2\times W_q\times\mathbb{Z}_q\times\mathbb{Z}_q$. For $\alpha,\beta,\gamma>0$, let $\mathcal{M}_E^{\alpha,\beta,\gamma}(N)$ be the number of quadruples $(x,w,r,s)\in\mathbb{Z}^2\times\mathbb{Z}^2\times\mathbb{Z}\times\mathbb{Z}$ such that
\begin{align*}
& w\text{ is primitive},\quad (r,s)\ne(0,0),\quad \norm{x}\le\alpha N,\\
& \norm{rw}\le\beta N,
\quad \norm{sw}\le\gamma N,
\quad (x,w,r,s)\bmod q\in E.
\end{align*}
Then, for $N\ge q$, uniformly in $E$,
\begin{equation}\label{eq:unifiedlifting}
\mathcal{M}_E^{\alpha,\beta,\gamma}(N)
=\frac{16\eta_q}{q^4}\alpha^2\beta\gamma\,\#E\ N^4\log N
+O_{\alpha,\beta,\gamma}\!\left(\frac{\#E}{q^4}N^4\right).
\end{equation}

\emph{(ii)} Let $F\subseteq W_q\times\mathbb{Z}_q\times\mathbb{Z}_q$. For $\beta,\gamma>0$, let $\mathcal{N}_F^{\beta,\gamma}(N)$ be the number of triples $(w,r,s)\in\mathbb{Z}^2\times\mathbb{Z}\times\mathbb{Z}$ such that
\[ w\text{ is primitive},\quad (r,s)\ne(0,0),\quad
\norm{rw}\le\beta N,
\quad \norm{sw}\le\gamma N,
\quad (w,r,s)\bmod q\in F. \]
Then, for $N\ge q$, uniformly in $F$,
\begin{equation}\label{eq:pinnedlifting}
\mathcal{N}_F^{\beta,\gamma}(N)
=\frac{4\eta_q}{q^2}\beta\gamma\,\#F\ N^2\log N +O_{\beta,\gamma}\left(\frac{\#F}{q^2}N^2\right).
\end{equation}
Moreover,
\begin{equation}\label{eq:ANuniform}
\#A_N(q)=\frac{4\#B_q}{q^2}N^2 + O\biggl(\frac{\#B_q}{q}N\biggr).
\end{equation}
\end{lemma}

\begin{proof}
We first carry out the calculation for one residue class
$e=(x_0,w_0,r_0,s_0)\in\mathbb{Z}_q^2\times W_q\times
\mathbb{Z}_q\times\mathbb{Z}_q$. The number of choices for $x$ is
\[ \frac{4\alpha^2}{q^2}N^2+O_{\alpha}\left(\frac{N}{q}\right), \]
where we used $N\ge q$. For $a\in\mathbb{Z}_q$, write
\[ \mathcal{A}_a(Y):=\#\{n\in\mathbb{Z} \,:\, n\equiv a\!\!\!\pmod q,
\ |n|\le Y\}=\frac{2Y}{q}+O(1), \]
which holds uniformly in $a$ and $Y\ge0$. Since $(r,s)\ne(0,0)$, every contributing $w$ has $\norm{w}\le CN$, where $C:=\max\{\beta,\gamma\}$. For a fixed primitive $w\equiv w_0\pmod q$, the choices for $(r,s)$ are counted by
\[ \mathcal{A}_{r_0}\!\Bigl(\frac{\beta N}{\norm{w}}\Bigr)
\mathcal{A}_{s_0}\!\Bigl(\frac{\gamma N}{\norm{w}}\Bigr), \]
with the forbidden pair $(r,s)=(0,0)$ subtracted when $r_0=s_0=0$. Also
\[ \mathcal{A}_{r_0}(Y)\mathcal{A}_{s_0}(Z)
=\frac{4YZ}{q^2}+O\left(\frac{Y}{q}\right)+O\left(\frac{Z}{q}\right)+O(1). \]
We now check the error terms in this summation. Choose a representative $\widetilde w_0\in\mathbb{Z}^2$ of $w_0$ with $\norm{\widetilde w_0}\le q/2$. Every $w\equiv w_0\pmod q$ is of the form $\widetilde w_0+qz$. Hence, for $N\ge q$,
\[ \#\{w \,:\, 0<\norm{w}\le CN,\ w\equiv w_0\!\!\!\pmod q\}=O_{\beta,\gamma}\left(\frac{N^2}{q^2}\right). \]
Thus the forbidden pair $(r,s)=(0,0)$, when it occurs, and the summed contribution of the $O(1)$ term in the product formula are both $O_{\beta,\gamma}(N^2/q^2)$. Also, separating the possible term $z=0$, and for a constant $C'$ depending only on $C$,
\[ \sum_{\substack{0<\norm{w}\le CN\\ w\equiv w_0\!\!\!\pmod q}}\frac1{\norm{w}}
\lesssim_{\beta,\gamma} 1 + \frac1q\sum_{0<\norm{z}\le C'(N/q+1)}\frac1{\norm{z}}
\lesssim_{\beta,\gamma} 1 + \frac{N}{q^2}. \]
Therefore the terms $O(Y/q)+O(Z/q)$ contribute
\[ O_{\beta,\gamma}\biggl(\frac Nq \sum_{\substack{0<\norm{w}\le CN\\ w\equiv w_0\!\!\!\pmod q}}\frac1{\norm{w}}\biggr)
=O_{\beta,\gamma}\left(\frac Nq+\frac{N^2}{q^3}\right)
=O_{\beta,\gamma}\left(\frac{N^2}{q^2}\right), \]
where the last step uses $N\ge q$.
Thus the $(w,r,s)$ count in one residue class is
\begin{equation}\label{eq:wrsmain}
\frac{4\beta\gamma N^2}{q^2}
\sum_{\substack{0<\norm{w}\le CN\\ w\equiv w_0 \!\!\!\pmod q\\ w\text{ primitive}}} \frac1{\norm{w}^2} + O_{\beta,\gamma}\left(\frac{N^2}{q^2}\right).
\end{equation}

It remains to evaluate the primitive sum. Since $w_0\in W_q$, every common divisor $d$ of the two coordinates of a vector $w\equiv w_0\pmod q$ is coprime to $q$, and therefore $d^{-1}$ exists in $\mathbb{Z}_{q}$. By the M\"{o}bius inversion formula we have
\[ \mathbbm{1}_{w\,\mathrm{primitive}} = \sum_{d\mid \mathrm{gcd}(w_1,w_2)}\mu(d), \]
so that
\[ \sum_{\substack{0<\norm{w}\le CN\\ w\equiv w_0\!\!\!\pmod q \\ w\text{ primitive}}} \frac1{\norm{w}^2}
= \sum_{\substack{0<\norm{w}\leq CN\\ w\equiv w_0\!\!\!\pmod q}}\sum_{d\mid \mathrm{gcd}(w_1,w_2)}
\frac{\mu(d)}{\norm{w}^{2}}. \]
Note that the latter sum is over all (not necessarily primitive) $w$ divisible by $d$. By interchanging the sums and performing the substitution $w=dz$, we obtain
\[ \sum_{\substack{d\ge1\\d,q\text{ coprime}}}\frac{\mu(d)}{d^2} \sum_{\substack{0<\norm{z}\le CN/d\\ z\equiv d^{-1}w_0\!\!\!\pmod q}} \frac1{\norm{z}^2}. \]
For any residue class $a\in\mathbb{Z}_q^2$ and $T\ge0$,
\[ \#\{z\in\mathbb{Z}^2 \,:\, 0<\norm{z}\le T, \ z\equiv a\!\!\!\pmod q\}
=\frac{4T^2}{q^2}+O\left(\frac{T}{q}\right)+O(1), \]
uniformly in $a$ and $q$. Partial summation therefore gives, uniformly in $a$ and $q$,
\[ \sum_{\substack{0<\norm{z}\le X\\ z\equiv a \!\!\!\pmod q}} \frac1{\norm{z}^2}
= \frac8{q^2}\log X+O(1) \]
for $X\ge2$: the $O(T/q)+O(1)$ error contributes
\[ O\left(\frac1q\right)+O(1)+O\left(\frac1{qX}\right)+O\left(\frac1{X^2}\right)=O(1). \]
If $CN<2$, then the primitive sum and $\eta_q\log N$ are both $O_{\beta,\gamma}(1)$, so the desired estimate for the primitive sum is immediate. We may therefore assume $CN\ge2$. Put $D=\lfloor CN/2\rfloor$ and split the sum in $d$ at the number $D$.
In the range $d\le D$ we have $CN/d\ge2$, so this part of the sum contributes
\[ \frac8{q^2} \sum_{\substack{d\le D\\d,q\text{ coprime}}}\frac{\mu(d)}{d^2} (\log N+\log C-\log d)+O_{\beta,\gamma}(1). \]
The terms involving $\log C$, $\log d$ and the displayed error are $O_{\beta,\gamma}(1)$, since $\sum_{d\geq1}(1+\log d)/d^2<\infty$. Extending the truncated coefficient of $\log N$ to the infinite Euler product costs
\[ O_{\beta,\gamma}\biggl(\frac{\log N}{q^2}\sum_{d>D}\frac1{d^2}\biggr)=O_{\beta,\gamma}(1), \]
and the range $d>D$ contributes only $O_{\beta,\gamma}(1)$. Therefore
\begin{equation}\label{eq:primitivesum}
\sum_{\substack{0<\norm{w}\le CN\\ w\equiv w_0\!\!\!\pmod q\\ w\text{ primitive}}} \frac1{\norm{w}^2}
=\eta_q\log N+O_{\beta,\gamma}(1).
\end{equation}
Combining \eqref{eq:wrsmain} and \eqref{eq:primitivesum} gives
\[ \frac{4\eta_q}{q^2}\beta\gamma N^2\log N + O_{\beta,\gamma}(N^2/q^2) \]
for the $(w,r,s)$ count in one residue class $(w_0,r_0,s_0)$. Summing over the residue classes in $F$ proves \eqref{eq:pinnedlifting}.

Multiplying the one-residue-class $(w,r,s)$ count by the $x$ count gives the stated main term. The product of the main $x$ term with the $(w,r,s)$ error is $O_{\alpha,\beta,\gamma}(N^4/q^4)$, while the product of the $x$ error with the $(w,r,s)$ main term is
\[ O_{\alpha,\beta,\gamma}\!\left(\frac{\eta_q N^3\log N}{q^3}\right)
=O_{\alpha,\beta,\gamma}\!\left(\frac{N^4}{q^4}\right), \]
since $\eta_q\le 8/q^2$ and $\log N\le Nq$. The remaining product of errors is also $O_{\alpha,\beta,\gamma}(N^4/q^4)$ because $N\ge q$. Thus the one-residue-class count is
\[ \frac{16\eta_q}{q^4}\alpha^2\beta\gamma N^4\log N + O_{\alpha,\beta,\gamma}\!\left(\frac{N^4}{q^4}\right). \]
Summing over $e\in E$ proves \eqref{eq:unifiedlifting}.

Finally, for each residue class $a\in\mathbb{Z}_q^2$,
\[ \#\{x\in\mathbb{Z}^2 \,:\, \norm{x}\le N,\ x\equiv a\!\!\!\pmod q\}
=\frac{4N^2}{q^2}+O\left(\frac{N}{q}\right), \]
because $N\ge q$. Summing over $a\in B_q$ proves \eqref{eq:ANuniform}.
\end{proof}

\subsection{The asymptotic lower bounds}
\begin{proof}[Proof of Proposition~\ref{prop:fixedmodulus}]
Fix $P$, $q$ and $B_q$ as in the proposition statement. 
We shall choose $C_0\ge1$ at the end. Throughout we also use the elementary lower bound
\begin{equation}\label{eq:etaqlower}
\eta_q=\frac8{q^2}\prod_{p\nmid q}\left(1-\frac1{p^2}\right) \ge \frac{8}{\zeta(2)q^2}.
\end{equation}

We first lift rectangles. Take a tuple counted by
$\mathcal{M}_{\Box_{\mathbb{Z}_q}(B_q)}^{1/4,1/4,1/4}(N)$. The residue condition says that
\[ x, \quad x+rw, \quad x+sJw, \quad x+rw+sJw \]
all have residues in $B_q$. The size restrictions give
\[ \norm{x}\le \frac{N}{4}, \quad \norm{rw}\le \frac{N}{4}, \quad \norm{sJw}=\norm{sw}\le \frac{N}{4}, \]
and hence all four displayed points lie in $[-N,N]^2$. They therefore lie in $A_N(q)$. Moreover,
\[ (x+rw)+(x+sJw)=x+(x+rw+sJw), \quad (rw)\cdot(sJw)=0, \]
so these four points form an ordered rectangle in the sense of the definition of $\Box(A_N(q))$. Since $(r,s)\ne(0,0)$ and $w$ is primitive, the two side vectors are not both zero. Conversely, a fixed ordered rectangle arising in this way has at most two primitive-direction representations of the form $rw$ and $sJw$, namely those corresponding to $w$ and $-w$: if $rw\ne0$, then $w$ is the primitive direction of $rw$ up to sign, while if $rw=0$ then $sJw\ne0$ and $w$ is determined up to sign by $sJw$. Hence Lemma~\ref{lm:lifting}, with $(\alpha,\beta,\gamma)=(1/4,1/4,1/4)$, gives the uniform lower bound
\[ \#\Box(A_N(q)) \ge \frac{\eta_q}{32q^4}\#\Box_{\mathbb{Z}_q}(B_q)N^4\log N - O\left(\frac{\#\Box_{\mathbb{Z}_q}(B_q)}{q^4}N^4\right). \]
Using \eqref{eq:etaqlower}, this can also be written as
\begin{equation}\label{eq:rectrelative}
\#\Box(A_N(q)) \ge \frac{\eta_q}{32q^4}\#\Box_{\mathbb{Z}_q}(B_q)N^4\log N \left(1-O\Bigl(\frac{q^2}{\log N}\Bigr)\right).
\end{equation}

We next control the unpinned triangle count. Let
$(k_1,k_2,k_3)\in\triangle(A_N(q))$ and put
\[ d=k_2-k_1, \quad h=2k_3-k_1-k_2. \]
Then $d\cdot h=0$. The case $d=h=0$ forces $k_1=k_2=k_3$ and contributes at most $\#A_N(q)\le (2N+1)^2$ triples. In every other case there is a primitive vector $w\in\mathbb{Z}^2$ and integers $r,s$, unique up to simultaneous sign change, such that
\[ d=rw, \quad h=sJw. \]
Indeed, if $d\ne0$ we take $w$ to be the primitive vector in the direction of $d$; since $h$ is an integer vector perpendicular to the primitive vector $w$, it is an integer multiple of $Jw$. The case $d=0$ and $h\ne0$ is the same argument applied to $h$ and $Jw$. Thus the possibilities $d=0$ and $h=0$ are included by allowing $r=0$ and $s=0$, respectively. The norm bounds are
\[ \norm{k_1}\le N, \quad \norm{rw}=\norm{k_2-k_1}\le2N, \quad \norm{sw}=\norm{h}\le4N. \]
Since $w$ is primitive, $w\bmod q\in W_q$. Also, modulo $q$,
\[ k_1, \quad k_1+rw=k_2, \quad k_1+\frac{rw+sJw}{2}=k_3 \]
all lie in $B_q$. Hence the two choices $(w,r,s)$ and $(-w,-r,-s)$ give two tuples counted by
$\mathcal{M}_{\triangle_{\mathbb{Z}_q}(B_q)}^{1,2,4}(N)$. The modular count may include additional tuples for which $rw+sJw$ is not divisible by $2$ in $\mathbb{Z}^2$, but that only enlarges the count. Therefore
\[ \#\triangle(A_N(q)) \le \frac12\mathcal{M}_{\triangle_{\mathbb{Z}_q}(B_q)}^{1,2,4}(N)+O(N^2). \]
Applying Lemma~\ref{lm:lifting} with $(\alpha,\beta,\gamma)=(1,2,4)$ yields
\[ \#\triangle(A_N(q)) \le \frac{64\eta_q}{q^4}\#\triangle_{\mathbb{Z}_q}(B_q)N^4\log N + O\left(\frac{\#\triangle_{\mathbb{Z}_q}(B_q)}{q^4}N^4\right) + O(N^2). \]
The first error term is the main term times $O(q^2/\log N)$, by \eqref{eq:etaqlower}. The degenerate contribution $O(N^2)$ is absorbed by the same relative error: Lemma~\ref{lm:crt} gives $\#\triangle_{\mathbb{Z}_q}(B_q)\gtrsim q^3$, and so, using \eqref{eq:etaqlower} and $N\ge q$,
\[ N^2=O\left(\frac{\eta_q}{q^4}\#\triangle_{\mathbb{Z}_q}(B_q)N^4\log N\cdot\frac{q^2}{\log N}\right). \]
Thus
\begin{equation}\label{eq:triasymp}
\#\triangle(A_N(q)) \le \frac{64\eta_q}{q^4}\#\triangle_{\mathbb{Z}_q}(B_q)N^4\log N \left(1+O\Bigl(\frac{q^2}{\log N}\Bigr)\right).
\end{equation}

It remains to estimate the pinned triangle count. Fix $k\in A_N(q)$ and put $k_0=k\bmod q\in B_q$. If $(k_1,k_2,k)\in\triangle^k(A_N(q))$, set
\[ d=k_2-k_1, \quad h=2k-k_1-k_2. \]
The case $d=h=0$ contributes only the pair $k_1=k_2=k$. In every other case the same primitive-direction argument gives $w\in\mathbb{Z}^2$ primitive and $r,s\in\mathbb{Z}$, unique up to simultaneous sign change, such that
\[ d=rw, \quad h=sJw, \quad \norm{rw}\le2N, \quad \norm{sw}\le4N. \]
Modulo $q$ the two base points can be recovered from the apex by
\[ k_1=k-\frac{rw+sJw}{2}, \quad k_2=k+\frac{rw-sJw}{2}, \]
and they lie in $B_q$. Thus, again allowing the harmless overcount from parity, we have
\[ \#\triangle^k(A_N(q)) \le \frac12\mathcal{N}_{\triangle^{k_0}_{\mathbb{Z}_q}(B_q)}^{2,4}(N)+1. \]
Lemma~\ref{lm:lifting}, now with $(\beta,\gamma)=(2,4)$, gives, uniformly in $k\in A_N(q)$,
\[ \#\triangle^k(A_N(q)) \le \frac{16\eta_q}{q^2}\#\triangle^{k_0}_{\mathbb{Z}_q}(B_q)N^2\log N + O\left(\frac{\#\triangle^{k_0}_{\mathbb{Z}_q}(B_q)}{q^2}N^2\right)+1. \]
The displayed error term is the main term times $O(q^2/\log N)$, by \eqref{eq:etaqlower}. The final $+1$ is also absorbed uniformly: Lemma~\ref{lm:crt} gives $\max_{b\in B_q}\#\triangle^b_{\mathbb{Z}_q}(B_q)\gtrsim q^2$, and hence, using \eqref{eq:etaqlower} and $N\ge q$,
\[ 1=O\left(\frac{\eta_q}{q^2}\max_{b\in B_q}\#\triangle^b_{\mathbb{Z}_q}(B_q)N^2\log N\cdot\frac{q^2}{\log N}\right). \]
Consequently, again uniformly in $k\in A_N(q)$,
\begin{equation}\label{eq:pinnedasymp}
\max_{k\in A_N(q)}\#\triangle^k(A_N(q)) \le \frac{16\eta_q}{q^2}\max_{b\in B_q}\#\triangle^b_{\mathbb{Z}_q}(B_q)N^2\log N \left(1+O\Bigl(\frac{q^2}{\log N}\Bigr)\right).
\end{equation}
Finally, Lemma~\ref{lm:lifting} also gives
\begin{equation}\label{eq:sizeANupper}
\#A_N(q) = \frac{4\#B_q}{q^2}N^2+O\left(\frac{\#B_q}{q}N\right)
= \frac{4\#B_q}{q^2}N^2\left(1+O\Bigl(\frac{q}{N}\Bigr)\right).
\end{equation}

Choosing $C_0$ sufficiently large in \eqref{eq:fixeducondition} makes the relative error terms $O(q^2/\log N)$ and $O(q/N)$ arbitrarily small, so that they can be absorbed into the corresponding dominant terms.
Combining the preceding estimates, we obtain
\[ \mathfrak{Q}(A_N(q)) \gtrsim \frac{\#\Box_{\mathbb{Z}_q}(B_q)}{\#\triangle_{\mathbb{Z}_q}(B_q)}, \]
and
\[ \mathfrak{Q}^{\mathrm{pin}}(A_N(q)) \gtrsim \frac{\#\Box_{\mathbb{Z}_q}(B_q)}{\#B_q\,\max_{b\in B_q}\#\triangle^b_{\mathbb{Z}_q}(B_q)}. \]
Applying Lemma~\ref{lm:crt} proves \eqref{eq:fixeduniform}.
\end{proof}

We can now finalise the proof of our main combinatorial result.

\begin{proof}[Proof of Theorem~\ref{thm:combinatorics}]
Bounds for $\mathfrak{R}(N)$ and $\mathfrak{R}^{\mathrm{pin}}(N)$ are now almost immediate.
Let $p_1<p_2<\cdots$ be the first $m$ primes congruent to $3$ modulo $4$, and put $q_m=p_1p_2\cdots p_m$. By the prime number theorem in arithmetic progressions \cite[\S20--22]{Davenport}, \cite[\S7.5]{Dujella}, there is an absolute constant $C_1$ such that
\[ q_m\le \exp(C_1m\log m) \]
for $m\geq2$.
Choosing
\[ m(N):=\left\lfloor c_1\frac{\log\log N}{\log\log\log N}\right\rfloor, \]
where $c_1>0$ is a sufficiently small absolute constant, gives $q_m^2 \leq (\log N)^{\theta}$ with $\theta < 1$. Then, for all sufficiently large $N$, the modulus $q_{m(N)}$ satisfies $N\geq \exp(C_0 q_{m(N)}^2)$, where $C_0$ is the constant from Proposition~\ref{prop:fixedmodulus}.
Applying that proposition with $q=q_{m(N)}$ gives
\[ \mathfrak{R}(N),\ \mathfrak{R}^{\mathrm{pin}}(N) \gtrsim \Bigl(\frac{4}{3}\Bigr)^{m(N)}, \]
which is the claimed bound.

Now we turn to the bounds for $\widetilde{\mathfrak{R}}(n)$ and $\widetilde{\mathfrak{R}}^{\mathrm{pin}}(n)$.
Again let $q=q_m=p_1p_2\cdots p_m$ be the product of the first $m$ primes congruent to $3$ modulo $4$. There is an absolute constant $C_2$ such that, whenever
\begin{equation}\label{eq:cardcondition}
n\ge e^{C_2q^2},
\end{equation}
there is an $n$-point set $A\subseteq\mathbb{Z}^2$ satisfying
\begin{equation}\label{eq:cardfixed}
\mathfrak{Q}(A),\ \mathfrak{Q}^{\mathrm{pin}}(A) \gtrsim \Bigl(\frac{4}{3}\Bigr)^m.
\end{equation}
Indeed, choose a nonnegative integer $N$ maximal subject to
$\#A_N(q)\le n$, and then choose an intermediate set
\[ A_N(q)\subseteq A\subseteq A_{N+1}(q), \quad \#A=n. \]
The individual configuration counts are monotone under inclusion. Thus
\[ \mathfrak{Q}(A) \ge \frac{\#\Box(A_N(q))}{\#\triangle(A_{N+1}(q))}, \]
and
\[ \mathfrak{Q}^{\mathrm{pin}}(A) \ge \frac{\#\Box(A_N(q))}{\#A_{N+1}(q)\,\max_{k\in A_{N+1}(q)}\limits \#\triangle^k(A_{N+1}(q))}. \]
We now spell out the uniform estimates needed for these two right-hand sides. The estimates \eqref{eq:rectrelative}, \eqref{eq:triasymp}, \eqref{eq:pinnedasymp} and \eqref{eq:sizeANupper}, applied with $M$ in place of $N$, imply the following: after increasing an absolute constant $C$ if necessary, whenever $M\ge \exp(C q^2)$,
\begin{align*}
\#\Box(A_M(q))
&\ge \frac{\eta_q}{64q^4}\#\Box_{\mathbb{Z}_q}(B_q)M^4\log M,\\
\#\triangle(A_M(q))
&\le \frac{128\eta_q}{q^4}\#\triangle_{\mathbb{Z}_q}(B_q)M^4\log M,\\
\max_{k\in A_M(q)}\#\triangle^k(A_M(q))
&\le \frac{32\eta_q}{q^2}\max_{b\in B_q}\#\triangle^b_{\mathbb{Z}_q}(B_q)M^2\log M,\\
\#A_M(q)
&\le \frac{8\#B_q}{q^2}M^2.
\end{align*}
The maximality of $N$ gives $n\le \#A_{N+1}(q)\le (2N+3)^2$, hence $N\ge \sqrt n/4$ for all large $n$. After increasing $C_2$, condition \eqref{eq:cardcondition} ensures $N\ge\exp(Cq^2)$, and the same bound with $N+1$ in place of $N$. Applying the displayed estimates with $M=N$ in the numerator and with $M=N+1$ in the denominator gives
\[ \mathfrak{Q}(A) \gtrsim \frac{\#\Box_{\mathbb{Z}_q}(B_q)N^4\log N}{\#\triangle_{\mathbb{Z}_q}(B_q)(N+1)^4\log(N+1)}
\gtrsim \frac{\#\Box_{\mathbb{Z}_q}(B_q)}{\#\triangle_{\mathbb{Z}_q}(B_q)}
=\Bigl(\frac43\Bigr)^m, \]
and
\begin{align*}
\mathfrak{Q}^{\mathrm{pin}}(A)
& \gtrsim \frac{\#\Box_{\mathbb{Z}_q}(B_q)N^4\log N}
{\#B_q (N+1)^2 \max_{b\in B_q}\#\triangle^b_{\mathbb{Z}_q}(B_q)(N+1)^2\log(N+1)} \\
&\gtrsim \frac{\#\Box_{\mathbb{Z}_q}(B_q)}{\#B_q\,\max_{b\in B_q}\#\triangle^b_{\mathbb{Z}_q}(B_q)}
=\Bigl(\frac43\Bigr)^m.
\end{align*}
This proves \eqref{eq:cardfixed}.

Now choose 
\[ m(n):=\left\lfloor c_2\frac{\log\log n}{\log\log\log n}\right\rfloor \]
with $c_2>0$ sufficiently small. Then, $n\geq\exp(C_2 q_{m(n)}^2)$ holds for all sufficiently large $n$.
Applying \eqref{eq:cardfixed} with $q=q_{m(n)}$ gives
\[ \widetilde{\mathfrak{R}}(n),\ \widetilde{\mathfrak{R}}^{\mathrm{pin}}(n) \gtrsim \Bigl(\frac{4}{3}\Bigr)^{m(n)}, \]
as desired.
\end{proof}

%%%%

\section{Equivalence of the Euclidean and periodic formulations}
\label{sec:transference}
This section proves the quantitative Euclidean-to-periodic transference stated in Theorem~\ref{thm:quant-transference-intro}. We also record a converse argument since it clarifies the close relationship of the periodic Stein-type and Mizohata--Takeuchi inequalities to their Euclidean counterparts.

Throughout this section, $d\geq1$ is fixed. In the Euclidean estimates, ``suitable'' initial data means that $u_0\in \mathrm{L}^2(\mathbb{R}^d)$ has compactly supported Fourier transform, so that the representation \eqref{eq:Eucl-solution} is absolutely convergent. In the periodic estimates, suitable data are trigonometric polynomials, equivalently $u_0\in\mathrm{L}^2(\mathbb{T}^d)$ with finite Fourier support, so that \eqref{eq:per-solution} is a finite sum. We identify functions on $\mathbb{T}^n$ with their $1$-periodic representatives on $\mathbb{R}^n$ whenever this causes no ambiguity.

For a finite set $A\subseteq\mathbb{Z}^d$ and coefficients $a=(a_k)_{k\in A}$, it will be convenient to write
\[ E_Aa(x,t):=\sum_{k\in A}a_k e^{2\pi i(x\cdot k+t|k|^2)}, \]
noting that $E_Aa$ is the general solution to the periodic Schr\"{o}dinger equation for the class of initial data under consideration.
We will use the following elementary reduction several times. For fixed $A$, it is enough to prove the periodic estimates for continuous weights. Indeed, let $P_\varepsilon$ be a nonnegative smooth approximate identity on $\mathbb{T}^{d+1}$. Since $E_Aa$ is bounded, $P_\varepsilon\ast w\to w$ in $\mathrm{L}^1$ implies convergence of the common left-hand side of the forthcoming \eqref{eq:periodic-MT-finite} and \eqref{eq:periodic-Stein-finite}. On the other hand,
\[ \rho_{\mathbb{Z}^d}^{\ast}(P_\varepsilon\ast w)(x,k)
=\int_{\mathbb{T}^{d+1}}P_\varepsilon(y,s)\rho_{\mathbb{Z}^d}^{\ast}w(x-y-2sk,k)\dd y\dd s, \]
and hence
\[ \sup_{x\in\mathbb{T}^d}\rho_{\mathbb{Z}^d}^{\ast}(P_\varepsilon\ast w)(x,k)
\leq \sup_{x\in\mathbb{T}^d}\rho_{\mathbb{Z}^d}^{\ast}w(x,k). \]
We note in passing that one may avoid approximating on the left-hand side here if one instead uses the local constancy property of $E_Aa$ that comes from its finite Fourier support. Local constancy is routinely used in this way in weighted extension theory; see \cite{CIW}, for example.

\subsection{From Euclidean to periodic}
Our first goal is to prove the direction claimed in Theorem~\ref{thm:quant-transference-intro}. The proof is a somewhat standard wave-packet ``thickening'' argument.

Fix $N\geq1$ and a finite set $A\subseteq\mathbb{Z}^d\cap[-N,N]^d$. We first prove
\begin{equation}\label{eq:periodic-MT-finite}
\int_{\mathbb{T}}\int_{\mathbb{T}^d}|E_Aa(x,t)|^2w(x,t)\dd x\dd t
\leq \mathfrak{M}_{\mathbb{R}^d}(N+1)\sup_{\substack{x\in\mathbb{T}^d\\ k\in A}}\rho^{\ast}_{\mathbb{Z}^d}w(x,k)\sum_{k\in A}|a_k|^2,
\end{equation}
for every nonnegative weight $w$ on $\mathbb{T}^{d+1}$. We then prove the corresponding Stein-type estimate
\begin{equation}\label{eq:periodic-Stein-finite}
\int_{\mathbb{T}}\int_{\mathbb{T}^d}|E_Aa(x,t)|^2w(x,t)\dd x\dd t
\leq \mathfrak{S}_{\mathbb{R}^d}(N+1)\sum_{k\in A}|a_k|^2\sup_{x\in\mathbb{T}^d}\rho^{\ast}_{\mathbb{Z}^d}w(x,k).
\end{equation}
Since $E_Aa$ is the periodic Schr\"{o}dinger solution with initial data $u_0(x)=\sum_{k\in A}a_ke^{2\pi ix\cdot k}$ and $\|u_0\|_{\mathrm{L}^2(\mathbb{T}^d)}=\|a\|_{\ell^2(A)}$, inequalities \eqref{eq:periodic-MT-finite} and \eqref{eq:periodic-Stein-finite} imply \eqref{eq:MT-transference} and \eqref{eq:St-transference}, respectively, after taking the least admissible constants.

By the reduction above we may suppose that $w$ is continuous. Define
\[ X_w(v):=\sup_{x\in\mathbb{R}^d}\int_0^1w(x-2tv,t)\dd t,\quad v\in\mathbb{R}^d. \]
Because $w$ is continuous and periodic in $x$, the function $X_w$ is uniformly continuous on bounded subsets of $\mathbb{R}^d$; indeed,
\[ |X_w(v_1)-X_w(v_2)|
\leq \sup_{x\in\mathbb{R}^d}\int_0^1|w(x-2tv_1,t)-w(x-2tv_2,t)|\dd t. \]
Choose a Schwartz function $\varphi$ on $\mathbb{R}^d$ such that $\widehat{\varphi}$ is compactly supported and $\|\varphi\|_{\mathrm{L}^2(\mathbb{R}^d)}=1$. For $L\geq1$, put $\varphi_L(x)=\varphi(L^{-1}x)$, so that $L^{-d}|\widehat{\varphi_L}|^2$ is an approximate identity at the origin. For $L$ sufficiently large, the sets $k+\supp\widehat{\varphi_L}$, $k\in A$, are pairwise disjoint and are contained in $[-N-1,N+1]^d$. Define
\[ F_L(\xi):=\sum_{k\in A}a_k\widehat{\varphi_L}(\xi-k) \]
and let
\[ u_L(x,t):=\int_{\mathbb{R}^d}F_L(\xi)e^{2\pi i(x\cdot\xi+t|\xi|^2)}\dd\xi. \]
We apply the Euclidean estimates to $u_L$ with the Euclidean weight equal to $w(x,t)\mathbbm{1}_{[0,1]}(t)$.

First,
\begin{align*}
L^{-d}\int_0^1\int_{\mathbb{R}^d}|u_L(x,t)|^2w(x,t)\dd x\dd t
&=\int_{\mathbb{T}}\int_{\mathbb{T}^d}L^{-d}\sum_{j\in\mathbb{Z}^d}|u_L(x+j,t)|^2w(x,t)\dd x\dd t.
\end{align*}
For fixed $(x,t)\in\mathbb{T}^{d+1}$, the quantity $u_L(x+j,t)$ is the $j$-th Fourier coefficient, in the variable $\eta\in[-\tfrac12,\tfrac12]^d$, of
\[ \eta\mapsto \sum_{k\in A}a_k e^{2\pi i(x\cdot(\eta+k)+t|\eta+k|^2)}\widehat{\varphi_L}(\eta), \]
provided $L$ is large enough that $\supp\widehat{\varphi_L}\subseteq[-\tfrac12,\tfrac12]^d$. By Parseval's identity,
\begin{align*}
L^{-d}\sum_{j\in\mathbb{Z}^d}|u_L(x+j,t)|^2
&=\int_{[-\tfrac12,\tfrac12]^d}\left|\sum_{k\in A}a_k e^{2\pi i(x\cdot(\eta+k)+t|\eta+k|^2)}\right|^2L^{-d}|\widehat{\varphi_L}(\eta)|^2\dd\eta.
\end{align*}
Since $L^{-d}|\widehat{\varphi_L}|^2$ is an approximate identity of mass $1$, this converges uniformly in $(x,t)$ to $|E_Aa(x,t)|^2$. Hence
\begin{equation}\label{eq:euc-per-left-limit}
L^{-d}\int_0^1\int_{\mathbb{R}^d}|u_L(x,t)|^2w(x,t)\dd x\dd t
\longrightarrow
\int_{\mathbb{T}}\int_{\mathbb{T}^d}|E_Aa(x,t)|^2w(x,t)\dd x\dd t.
\end{equation}
Also, by Plancherel's theorem and the disjointness of the translated supports,
\begin{equation}\label{eq:euc-per-norm-limit}
L^{-d}\|u_L(\cdot,0)\|_{\mathrm{L}^2(\mathbb{R}^d)}^2
=L^{-d}\|F_L\|_{\mathrm{L}^2(\mathbb{R}^d)}^2
=\sum_{k\in A}|a_k|^2.
\end{equation}
Finally, if
\[ A_L:=\bigcup_{k\in A}(k+\supp\widehat{\varphi_L}), \]
then the uniform continuity of $X_w$ gives
\begin{equation}\label{eq:euc-per-xray-limit}
\sup_{v\in A_L}X_w(v)
\longrightarrow \max_{k\in A}X_w(k)
=\sup_{\substack{x\in\mathbb{T}^d\\ k\in A}}\rho^{\ast}_{\mathbb{Z}^d}w(x,k).
\end{equation}
Applying the Euclidean Mizohata--Takeuchi inequality with constant $\mathfrak{M}_{\mathbb{R}^d}(N+1)$ gives
\[ L^{-d}\int_0^1\int_{\mathbb{R}^d}|u_L(x,t)|^2w(x,t)\dd x\dd t
\leq \mathfrak{M}_{\mathbb{R}^d}(N+1)\sup_{v\in A_L}X_w(v)
L^{-d}\|u_L(\cdot,0)\|_{\mathrm{L}^2(\mathbb{R}^d)}^2. \]
Letting $L\to\infty$ and using \eqref{eq:euc-per-left-limit}, \eqref{eq:euc-per-norm-limit} and \eqref{eq:euc-per-xray-limit} proves \eqref{eq:periodic-MT-finite}.

For the Stein-type estimate, applying \eqref{eq:euc-Stein} with constant $\mathfrak{S}_{\mathbb{R}^d}(N+1)$ gives
\[ L^{-d}\int_0^1\int_{\mathbb{R}^d}|u_L(x,t)|^2w(x,t)\dd x\dd t
\leq \mathfrak{S}_{\mathbb{R}^d}(N+1)L^{-d}\int_{\mathbb{R}^d}|F_L(v)|^2X_w(v)\dd v. \]
The right-hand side converges to
\[ \mathfrak{S}_{\mathbb{R}^d}(N+1)
\sum_{k\in A}|a_k|^2\sup_{x\in\mathbb{T}^d}\rho^{\ast}_{\mathbb{Z}^d}w(x,k), \]
because the translated supports of $\widehat{\varphi_L}$ are disjoint and $X_w$ is uniformly continuous near $A$. Together with \eqref{eq:euc-per-left-limit}, this proves \eqref{eq:periodic-Stein-finite} and completes the proof of Theorem~\ref{thm:quant-transference-intro}.

\subsection{From periodic to Euclidean}\label{Sect:pertoEuc}
Suppose that the periodic Stein-type and Mizohata--Takeuchi inequalities \eqref{eq:per-Stein} and \eqref{eq:per-MT} hold, or equivalently that  \eqref{eq:periodic-MT-finite} and \eqref{eq:periodic-Stein-finite} hold with a constant $C$ in place of $\mathfrak{M}_{\mathbb{R}^d}(N+1)$ and $\mathfrak{S}_{\mathbb{R}^d}(N+1)$, respectively. Here we show that the global Euclidean Stein-type and Mizohata--Takeuchi inequalities \eqref{eq:euc-Stein-glo} and \eqref{eq:euc-MT-glo} may be recovered. Since these are known to be false (disproved by Cairo in \cite{Cairo}), this provides an alternative way to conclude the failure of the global periodic estimates \eqref{eq:per-Stein} and \eqref{eq:per-MT}. As we shall see, our argument relies heavily on scaling and approximation, meaning that we are unable to conclude a reverse form of the finitary (band-limited) Theorem~\ref{thm:quant-transference-intro}. 

We first record the scaling of the periodic estimates. For $T>0$, let
\[ Q_T:=[-T/2,T/2]^d\times[-T^2/2,T^2/2]. \]

\begin{lemma}
Let $A\subseteq\mathbb{Z}^d$ be finite, and let $w$ be a nonnegative function on $\mathbb{R}^d\times\mathbb{R}$ that is periodic with spatial period $T$ in each coordinate and temporal period $T^2$. If the finite-frequency periodic Mizohata--Takeuchi estimate holds with constant $C$, then
\begin{align}\label{eq:scaled-periodic-MT}
\int_{Q_T}&\bigg|\frac{1}{T^d}\sum_{\xi\in T^{-1}A}a_\xi e^{2\pi i(x\cdot\xi+t|\xi|^2)}\bigg|^2w(x,t)\dd x\dd t\notag\\
&\leq C\sup_{\substack{x\in\mathbb{R}^d\\ \xi\in T^{-1}A}}\int_{-T^2/2}^{T^2/2}w(x-2t\xi,t)\dd t\,\frac{1}{T^d}\sum_{\xi\in T^{-1}A}|a_\xi|^2.
\end{align}
Similarly, if the finite-frequency periodic Stein-type estimate holds with constant $C$, then
\begin{align}\label{eq:scaled-periodic-Stein}
\int_{Q_T}&\bigg|\frac{1}{T^d}\sum_{\xi\in T^{-1}A}a_\xi e^{2\pi i(x\cdot\xi+t|\xi|^2)}\bigg|^2w(x,t)\dd x\dd t\notag\\
&\leq \frac{C}{T^d}\sum_{\xi\in T^{-1}A}|a_\xi|^2\sup_{x\in\mathbb{R}^d}\int_{-T^2/2}^{T^2/2}w(x-2t\xi,t)\dd t.
\end{align}
\end{lemma}

\begin{proof}
We prove \eqref{eq:scaled-periodic-Stein}; the proof of \eqref{eq:scaled-periodic-MT} is the same with the supremum taken outside the sum. Define
\[ w_T(y,s):=T^2w(Ty,T^2s), \]
which is $1$-periodic in $(y,s)$. Changing variables $x=Ty$, $t=T^2s$ and writing $m=T\xi$, we obtain
\begin{align*}
& \int_{Q_T} \bigg|\frac{1}{T^d}\sum_{\xi\in T^{-1}A}a_\xi e^{2\pi i(x\cdot\xi+t|\xi|^2)}\bigg|^2 w(x,t)\dd x\dd t\\
& = \frac{1}{T^d} \int_{\mathbb{T}}\int_{\mathbb{T}^d} \bigg|\sum_{m\in A}a_{m/T}e^{2\pi i(y\cdot m+s|m|^2)}\bigg|^2 w_T(y,s)\dd y\dd s.
\end{align*}
Applying the finite-frequency periodic Stein-type estimate with constant $C$ to the last integral gives
\[ \frac{C}{T^d}\sum_{m\in A}|a_{m/T}|^2\sup_{y\in\mathbb{T}^d}\int_{\mathbb{T}}w_T(y-2sm,s)\dd s. \]
The last X-ray transform scales as
\[ \sup_{y\in\mathbb{T}^d}\int_{\mathbb{T}}w_T(y-2sm,s)\dd s
=\sup_{x\in\mathbb{R}^d}\int_{-T^2/2}^{T^2/2}w(x-2t(m/T),t)\dd t, \]
where the interval of integration may be centred because the integrand is $1$-periodic in $s$. This proves \eqref{eq:scaled-periodic-Stein}.
\end{proof}

We now pass to the Euclidean limit. 
Evidently it suffices to establish the manifestly stronger global inequalities \eqref{eq:euc-Stein-glo} and \eqref{eq:euc-MT-glo} from their periodic counterparts \eqref{eq:per-Stein} and \eqref{eq:per-MT}. By the parabolic scale invariance of \eqref{eq:euc-Stein-glo} and \eqref{eq:euc-MT-glo} it suffices to consider initial data with $\supp\widehat{u_0}\subseteq B(0,1)$.
We may also restrict attention to continuous compactly supported (Euclidean) weights $w$, as the general case then follows by standard approximation arguments.

Next we observe that we may further assume that $\widehat{u_0}$ is smooth. If $\widehat{u_{n,0}}\to\widehat{u_0}$ in $\mathrm{L}^2$ with $\supp\widehat{u_{n,0}}\subseteq B(0,1)$ for all $n$, and
\[ u_n(x,t):=\int_{\mathbb{R}^d}\widehat{u_{n,0}}(\xi)e^{2\pi i(x\cdot\xi+t|\xi|^2)}\dd\xi, \]
then $u_n\to u$ uniformly, and so
\[ \int_{\mathbb{R}}\int_{\mathbb{R}^d}|u_n(x,t)|^2w(x,t)\dd x\dd t
\longrightarrow
\int_{\mathbb{R}}\int_{\mathbb{R}^d}|u(x,t)|^2w(x,t)\dd x\dd t \]
by the assumed regularity of $w$.
Similarly, the right-hand side of \eqref{eq:euc-Stein-glo} also converges suitably by the uniform continuity of the map
\begin{equation}\label{eq:Xraydef}
X_w(v):= \sup_{x\in\mathbb{R}^d}\int_{\mathbb{R}}w(x-2tv,t)\dd t,
\end{equation} 
recalling that $w$ is continuous and compactly supported. For the right-hand side of \eqref{eq:euc-MT-glo} we make the additional assumption, as we may, that the support of $\widehat{u_{n,0}}$ is decreasing to that of $\widehat{u_0}$ as $n\rightarrow\infty$, so that
\[ \sup_{\substack{x\in\mathbb{R}^d\\ v\in\supp\widehat{u_{n,0}}}}\!\!\!\!\rho^{\ast}w(x,v)\rightarrow \sup_{\substack{x\in\mathbb{R}^d\\ v\in\supp\widehat{u_0}}}\!\!\!\!\rho^{\ast}w(x,v), \]
again by the uniform continuity of $X_w$.
We may therefore assume that $\widehat{u_0}$ is supported in $B(0,1)$ and smooth.

Fix a continuous nonnegative $w$ supported in $[-R,R]^{d+1}$, for some $R>0$.
For $T>0$ define the full space-time periodisation
\[ w^T(x,t):=\sum_{j\in\mathbb{Z}^d}\sum_{l\in\mathbb{Z}}w(x+Tj,t+T^2l). \]
For $T$ sufficiently large, $w^T=w$ on $Q_T$. Moreover, if $v\in B(0,1)\cap T^{-1}\mathbb{Z}^d$, then the compatibility of the velocity $v=m/T$ with the periods $T$ and $T^2$, together with the support of $w$, gives
\begin{equation}\label{eq:remaining}
\sup_{x\in\mathbb{R}^d} \int_{-T^2/2}^{T^2/2} w^T(x-2tv,t)\dd t=X_w(v)
\end{equation}
for all sufficiently large $T$, uniformly in such $v$.
This follows from the support hypothesis on $w$, which means that the sum
\[ \int_{-T^2/2}^{T^2/2} w^T(x-2tv,t) \dd t=\int_{-R}^{R}w^T(x-2tv,t)\dd t=\sum_{j\in\mathbb{Z}^d}\int_{-R}^{R}w(x+Tj-2tv,t)\dd t \]
has at most one nonzero term for each $x\in\mathbb{R}^d$, $|v|\leq 1$ and $T$ sufficiently large compared with $R$. This reduces \eqref{eq:remaining} to the definition \eqref{eq:Xraydef}.
%to showing that
%\[ \sup_{x\in\mathbb{R}^d} \int_{-R}^{R}w(x-2tv,t)\dd t= X_w(v), \]
%which is, of course, evident.

The Riemann sums
\[ U_T(x,t):=\frac{1}{T^d}\sum_{\xi\in T^{-1}\mathbb{Z}^d}\widehat{u_0}(\xi)e^{2\pi i(x\cdot\xi+t|\xi|^2)} \]
are finite and converge uniformly to
\[ u(x,t)=\int_{\mathbb{R}^d}\widehat{u_0}(\xi)e^{2\pi i(x\cdot\xi+t|\xi|^2)}\dd\xi \]
on the compact support of $w$. Hence
\begin{equation}\label{eq:r-sum-left}
\int_{\mathbb{R}}\int_{\mathbb{R}^d}|u(x,t)|^2w(x,t)\dd x\dd t
=\lim_{T\to\infty}\int_{Q_T}|U_T(x,t)|^2w^T(x,t)\dd x\dd t.
\end{equation}
Applying \eqref{eq:scaled-periodic-Stein} to the right-hand side and using \eqref{eq:remaining}, we get
\begin{align*}
\int_{\mathbb{R}}\int_{\mathbb{R}^d}|u(x,t)|^2w(x,t)\dd x\dd t
&\leq C\limsup_{T\to\infty}\frac{1}{T^d}\sum_{\xi\in T^{-1}\mathbb{Z}^d}|\widehat{u_0}(\xi)|^2X_w(\xi)\\
&= C\int_{\mathbb{R}^d}|\widehat{u_0}(v)|^2\sup_{x\in\mathbb{R}^d}\int_{\mathbb{R}}w(x-2tv,t)\dd t\dd v,
\end{align*}
establishing \eqref{eq:euc-Stein-glo}.

The Mizohata--Takeuchi estimate \eqref{eq:euc-MT-glo} follows in a similar way from \eqref{eq:scaled-periodic-MT}. Indeed,
\[ \frac{1}{T^d}\sum_{\xi\in T^{-1}\mathbb{Z}^d}|\widehat{u_0}(\xi)|^2\longrightarrow \|u_0\|_{\mathrm{L}^2(\mathbb{R}^d)}^2, \]
and, by the continuity of $X_w$ on $B(0,1)$,
\[ \sup_{\xi\in T^{-1}\mathbb{Z}^d\cap\,\supp\widehat{u_0}}X_w(\xi)
\longrightarrow \sup_{v\in\supp\widehat{u_0}}X_w(v). \]
The inequality \eqref{eq:euc-MT-glo} now follows
from \eqref{eq:scaled-periodic-MT} applied to \eqref{eq:r-sum-left}.

%%%%%

\section{Periodic constants from configuration lower bounds}
\label{sec:config-lower}

We now prove Theorem~\ref{thm:config-lower-intro}. 
Let $A\subseteq\mathbb{Z}^2$ be a nonempty and finite set of frequencies.
Then
\[ u(x,t) = u_A(x,t) := \sum_{k\in A} e^{2\pi i(x\cdot k+t|k|^2)} \]
is the solution \eqref{eq:per-solution} of the periodic Schr\"{o}dinger equation with the initial data
\[ u_0(x) = u_A(x,0) = \sum_{k\in A} e^{2\pi i x\cdot k}, \]
which clearly has Fourier coefficients 
\[ \widehat{u_0}(k) = \begin{cases}
1 & \text{if } k\in A,\\
0 & \text{if } k\not\in A.
\end{cases} \]

\begin{lemma}\label{lem:counts}
For a finite $A\subseteq\mathbb{Z}^2$,
\begin{equation}\label{eq:L4rect}
\|u_A\|_{\mathrm{L}^4(\mathbb{T}^2\times\mathbb{T})}^4 = \#\Box(A).
\end{equation}
Moreover, for every $k\in A$,
\begin{equation}\label{eq:Xrayiso}
\sup_{x\in\mathbb{T}^2}\rho^{\ast}_{\mathbb{Z}^2}(|u_A|^2)(x,k) = \#\triangle^k(A).
\end{equation}
\end{lemma}

\begin{proof}
Expanding the fourth power and integrating,
\begin{align*} 
\|u_A\|_{\mathrm{L}^4(\mathbb{T}^2\times\mathbb{T})}^4 =   \int_{\mathbb{T}} \int_{\mathbb{T}^2} \sum_{k_1,k_2,k_3,k_4\in A} & e^{2\pi i(x\cdot k_1+t|k_1|^2)} e^{-2\pi i(x\cdot k_2+t|k_2|^2)} \\
& e^{2\pi i(x\cdot k_3+t|k_3|^2)} e^{-2\pi i(x\cdot k_4+t|k_4|^2)} \dd x \dd t, 
\end{align*}
gives precisely the number of quadruples $(k_1,k_2,k_3,k_4)$ in $A^4$ satisfying
\[ k_1+k_3=k_2+k_4, \quad |k_1|^2+|k_3|^2=|k_2|^2+|k_4|^2. \]
Writing $k_3=k_2+k_4-k_1$, the second condition becomes
\[ (k_2-k_1)\cdot(k_4-k_1)=0. \]
This identifies these quadruples with the ordered rectangles in $\Box(A)$, proving \eqref{eq:L4rect}.

For the second part, observe
\begin{align*}
\rho^{\ast}_{\mathbb{Z}^2}(|u_A|^2)(x,k)
& =\int_0^1 \sum_{k_1,k_2\in A} e^{2\pi i((x-2tk)\cdot k_1+t|k_1|^2)} e^{-2\pi i((x-2tk)\cdot k_2+t|k_2|^2)} \dd t \\
& =\sum_{\substack{k_1,k_2\in A\\
|k_1|^2-|k_2|^2-2k\cdot(k_1-k_2)=0}}e^{2\pi ix\cdot(k_1-k_2)}.
\end{align*}
The last sum is real-valued; it is bounded by the total number of surviving pairs $(k_1,k_2)$ and equality is attained at $x=0$. The summing condition can be rewritten as
\[ (k_1-k_2)\cdot(k_1+k_2-2k)=0. \]
This proves \eqref{eq:Xrayiso}.
\end{proof}

\begin{proof}[Proof of Theorem~\ref{thm:config-lower-intro}]
We take $u=u_A$ and $u_0$ as before, and choose the weight
\begin{equation}\label{eq:2Dweight}
w=|u_A|^2.
\end{equation}
The periodic Stein-type and Mizohata--Takeuchi estimates \eqref{eq:per-Stein} and \eqref{eq:per-MT} applied to $u=u_A$ then read
\begin{equation}\label{eq:per-Stein2}
\|u_A\|_{\mathrm{L}^4(\mathbb{T}^2\times\mathbb{T})}^4 \leq C \sum_{k\in A} \sup_{x\in\mathbb{T}^2} \rho^{\ast}_{\mathbb{Z}^2}(|u_A|^2)(x,k)
\end{equation}
and
\begin{equation}\label{eq:per-MT2}
\|u_A\|_{\mathrm{L}^4(\mathbb{T}^2\times\mathbb{T})}^4 \leq C \sup_{\substack{x\in\mathbb{T}^2\\ k\in A}} \rho^{\ast}_{\mathbb{Z}^2}(|u_A|^2)(x,k) \,\|u_0\|_{\mathrm{L}^2(\mathbb{T}^2)}^2.
\end{equation}
Applying Lemma~\ref{lem:counts}, observing $\|u_0\|_{\mathrm{L}^2(\mathbb{T}^2)}^2=\#A$ and varying $A$, inequality \eqref{eq:per-MT2} becomes
\begin{align*} 
\#\Box(A) \leq \mathfrak{M}_{\mathbb{T}^2}(N) \,\#A \,\max_{k\in A} \#\triangle^k(A) & \quad \text{if } A\subseteq\mathbb{Z}^2\cap[-N,N]^2, \\
\#\Box(A) \leq \widetilde{\mathfrak{M}}_{\mathbb{T}^2}(n) \,\#A \,\max_{k\in A} \#\triangle^k(A) & \quad \text{if } A\subseteq\mathbb{Z}^2,\ \#A=n. 
\end{align*}
Dividing by the quantities on the right-hand sides that depend on $A$ and taking suprema over $A$, we deduce \eqref{eq:MT-config-lower-intro}.
Similarly, inequality \eqref{eq:per-Stein2}, together with
\[ \sum_{k\in A} \sup_{x\in\mathbb{T}^2}\rho^{\ast}_{\mathbb{Z}^2}(|u_A|^2)(x,k) = \#\triangle(A), \]
implies
\begin{align*} 
\#\Box(A) \leq \mathfrak{S}_{\mathbb{T}^2}(N) \,\#\triangle(A) & \quad \text{if } A\subseteq\mathbb{Z}^2\cap[-N,N]^2, \\
\#\Box(A) \leq \widetilde{\mathfrak{S}}_{\mathbb{T}^2}(n) \,\#\triangle(A) & \quad \text{if } A\subseteq\mathbb{Z}^2,\ \#A=n. 
\end{align*}
These two estimates immediately establish \eqref{eq:Stein-config-lower-intro}.
\end{proof}
We remark that tomographic identities for extension operators, similar to \eqref{eq:Xrayiso}, have been the subject of some investigation recently in \cite{BN,BNS,BGNO} in the setting of continuous (as opposed to discrete) extension operators.
%%%%%

\section{Trigonometric weight obstructions}
\label{sec:d1-log-weight}

This section proves Theorem~\ref{thm:d1-log-intro}. 
The one-dimensional weight can be built from subset sums of the space-time parabola points $(2^j,4^j)$; see Figure~\ref{fig:pts_parabola}. At every tested ``velocity,'' the X-ray transform sees only diagonal pairs of subset sums, while multiplication by the full one-dimensional parabola sum creates many repeated Fourier frequencies. 
The actual construction that we present can be thought of as a discrete variant of Cairo's counterexample~\cite{Cairo}. However, the setting of $\mathbb{T}^{d+1}$, once again, allows us to use purely combinatorial reasoning.
Also, since we aim to give a $d$-dimensional example that upgrades the $(\log N)$-growth to $(\log N)^d$, we actually need to use a variant of the tensor-product construction.

In the formulation where the constants are expressed in terms of $n=\#\supp\widehat{u_0}$, our example even exhibits the announced linear growth \eqref{eq:justlinear}. 
After proving this lower bound, we record a trivial but useful converse, as the Cauchy--Schwarz inequality and the average of the X-ray transform easily give the matching upper bound in \eqref{eq:justlinear}.

\subsection{Construction of the higher-dimensional weight}
Fix $d\ge1$ and a vector of positive integers $\mathbf{m}=(m_1,\ldots,m_d)$.
For each coordinate put $r_l:=\lfloor m_l/2\rfloor$. Let $\mathcal{J}_{l,s}$ be the family of $s$-element subsets of $\{1,2,\ldots,m_l\}$ and, for a $d$-tuple of indices $\mathbf{s}=(s_1,\ldots,s_d)$, write 
\[ \mathcal{J}_{\mathbf{s}}:=\mathcal{J}_{1,s_1}\times\cdots\times\mathcal{J}_{d,s_d}. \]
We also write $\mathbf{r}=(r_1,\ldots,r_d)$ and $\mathbf{1}=(1,\ldots,1)$.
Define
\[ M_l:=m_1+\cdots+m_{l-1} \]
for $l=1,\ldots,d$, with the convention $M_1=0$. The frequency set in our example will be the discrete Cartesian product
\[ A_{\mathbf{m}} := \bigl\{(2^{M_1+j_1},2^{M_2+j_2},\ldots,2^{M_d+j_d}) \,:\, j_l\in\{1,2,\ldots,m_l\}\text{ for every }l=1,\ldots,d\bigr\} \subseteq\mathbb{Z}^d. \]
Note that
\[ \#A_{\mathbf{m}}=m_1\cdots m_d \quad\text{and}\quad A_{\mathbf{m}}\subseteq[0,2^{m_1+\cdots+m_d}]^d. \]
Consider the solution \eqref{eq:per-solution} of the periodic Schr\"{o}dinger equation given by the formula
\begin{equation}\label{eq:logNsolution}
u(x,t) = u_{\mathbf{m}}(x,t) :=\sum_{k\in A_{\mathbf{m}}}e^{2\pi i(k\cdot x+|k|^2t)}
= \sum_{j_1=1}^{m_1} \cdots \sum_{j_d=1}^{m_d} e^{2\pi i\sum_{l=1}^d (2^{M_l+j_l}x_l+4^{M_l+j_l}t)}.
\end{equation}
Then $u_0(x) = u_{\mathbf{m}}(x,0)$ has frequency support precisely $A_{\mathbf{m}}$ and
\[ \|u_0\|_{\mathrm{L}^2(\mathbb{T}^d)}^2 = \#A_{\mathbf{m}} = m_1\cdots m_d. \]

For each $l=1,2,\ldots,d$ and every subset $J\subseteq\{1,2,\ldots,m_l\}$ put
\[ X_{l,J}:=\sum_{j\in J} 2^{M_l+j}, \quad Y_{l,J}:=\sum_{j\in J} 4^{M_l+j}; \]
see Figure~\ref{fig:pts_parabola} again.
If $\mathbf{J}=(J_1,\ldots,J_d)$ is a $d$-tuple of such subsets $J_l\subseteq\{1,2,\ldots,m_l\}$, define
\[ X_{\mathbf{J}}:=(X_{1,J_1},\ldots,X_{d,J_d})\in\mathbb{Z}^d. \]
Let us first define an auxiliary $1$-periodic trigonometric polynomial by
\[ f_{\mathbf{m}}(x,t) := \sum_{\mathbf{J}\in\mathcal{J}_{\mathbf{r}}} e^{2\pi i (X_{\mathbf{J}}\cdot x+ t\sum_{l=1}^dY_{l,J_l})} = \prod_{l=1}^d \Bigl( \sum_{J\in\mathcal{J}_{l,r_l}} e^{2\pi i(X_{l,J}x_l+Y_{l,J}t)} \Bigr), \]
and then define the desired weight as
\begin{equation}\label{eq:logNweight}
w(x,t) = w_{\mathbf{m}}(x,t) := |f_{\mathbf{m}}(x,t)|^2.
\end{equation}

\begin{figure}
\begin{center}
\begin{tikzpicture}[x=0.09cm, y=0.0015cm]
    \draw[->, thick] (0,0) -- (135,0) node[right] {$X_{1,J}$};
    \draw[->, thick] (0,0) -- (0, 5400) node[above] {$Y_{1,J}$};
    \draw[dashed, darkgray, very thick, domain=0:73.5, samples=100] plot (\x, {\x*\x});
    \node[darkgray, left] at (60, 4000) {$Y = X^2$};
    \foreach \x in {20, 40, 60, 80, 100, 120} {
        \draw[thick] (\x, 80) -- (\x, -80) node[below=4pt] {\x};
    }
    \foreach \y in {1000, 2000, 3000, 4000, 5000} {
        \draw[thick] (1.5, \y) -- (-1.5, \y) node[left=4pt] {\y};
    }
    \node[left=4pt] at (0,0) {0};
    \foreach \bA in {0,1} {
    \foreach \bB in {0,1} {
    \foreach \bC in {0,1} {
    \foreach \bD in {0,1} {
    \foreach \bE in {0,1} {
    \foreach \bF in {0,1} {
        \pgfmathtruncatemacro{\Sum}{\bA+\bB+\bC+\bD+\bE+\bF}
        \ifnum\Sum=3 
        \else
            \pgfmathsetmacro{\X}{\bA*2 + \bB*4 + \bC*8 + \bD*16 + \bE*32 + \bF*64}
            \pgfmathsetmacro{\Y}{\bA*4 + \bB*16 + \bC*64 + \bD*256 + \bE*1024 + \bF*4096}
            \fill[gray, opacity=0.8] (\X, \Y) circle (1.5pt);
        \fi
    }}}}}}
    \foreach \bA in {0,1} {
    \foreach \bB in {0,1} {
    \foreach \bC in {0,1} {
    \foreach \bD in {0,1} {
    \foreach \bE in {0,1} {
    \foreach \bF in {0,1} {
        \pgfmathtruncatemacro{\Sum}{\bA+\bB+\bC+\bD+\bE+\bF}
        \ifnum\Sum=3
            \pgfmathsetmacro{\X}{\bA*2 + \bB*4 + \bC*8 + \bD*16 + \bE*32 + \bF*64}
            \pgfmathsetmacro{\Y}{\bA*4 + \bB*16 + \bC*64 + \bD*256 + \bE*1024 + \bF*4096}
            \fill[black] (\X, \Y) circle (1.5pt);
        \fi
    }}}}}}
\end{tikzpicture}
\end{center}
\caption{Illustration of points $(X_{1,J},Y_{1,J})$ for $d=1$ and $m_1=6$. Darker points correspond to subsets $J\subseteq\{1,\ldots,6\}$ with $\#J=3$.}
\label{fig:pts_parabola}
\end{figure}
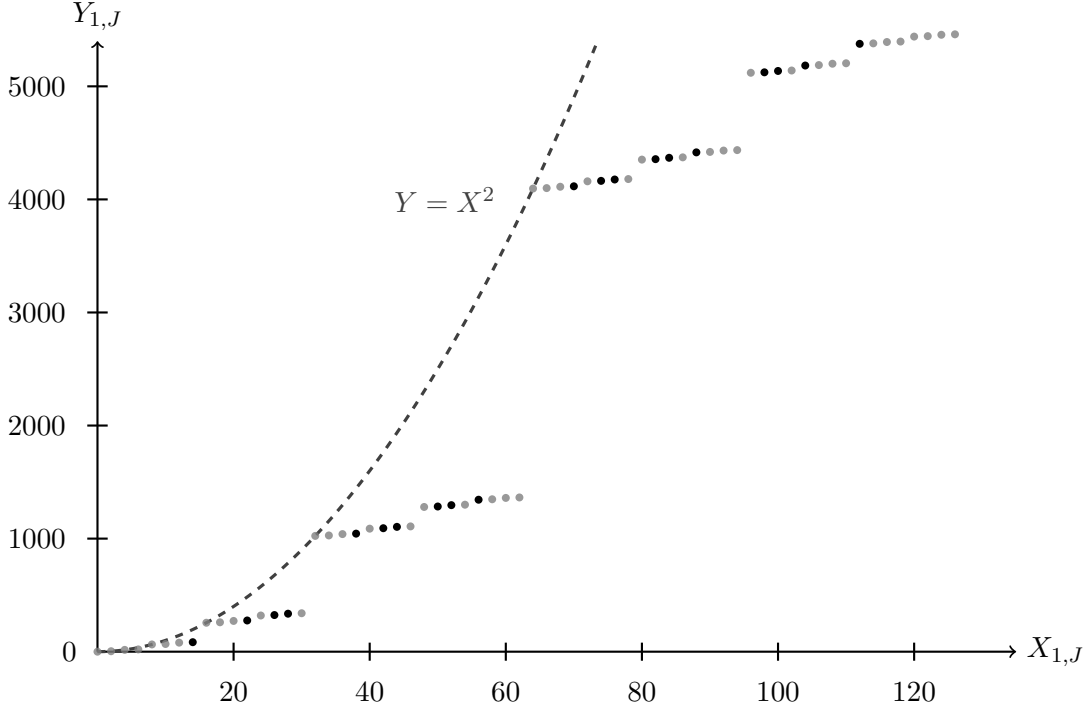

The weight $w_{\mathbf{m}}$ has the following properties.

\begin{lemma}\label{lm:dyadic-test-estimates}
For every $(x,k)\in\mathbb{T}^d\times A_{\mathbf{m}}$ we have
\begin{equation}\label{eq:xray-vector-exact}
\rho^{\ast}_{\mathbb{Z}^d}w_{\mathbf{m}}(x,k) = \prod_{l=1}^d\binom{m_l}{r_l}.
\end{equation}
Also,
\begin{equation}\label{eq:numerator-vector-lower}
\int_{\mathbb{T}} \int_{\mathbb{T}^d} |u_{\mathbf{m}}(x,t)|^2 w_{\mathbf{m}}(x,t)\dd x\dd t \ge \prod_{l=1}^d (r_l+1)^2\binom{m_l}{r_l+1}.
\end{equation}
\end{lemma}

\begin{proof}
We first prove \eqref{eq:xray-vector-exact}. Take 
\[ k = \bigl(2^{M_1+q_1},\ldots,2^{M_d+q_d}\bigr)\in A_{\mathbf{m}} \]
determined by some $q_l\in\{1,2,\ldots,m_l\}$; $l=1,\ldots,d$. Expanding $|f_{\mathbf{m}}(x,t)|^2$, we obtain the sum
\[ w_{\mathbf{m}}(x,t) = \sum_{\substack{\mathbf{J}=(J_1,\ldots,J_d)\in\mathcal{J}_{\mathbf{r}} \\ \mathbf{J}'=(J'_1,\ldots,J'_d)\in\mathcal{J}_{\mathbf{r}}}} e^{2\pi i\left((X_{\mathbf{J}}-X_{\mathbf{J}'})\cdot x+ t\sum_{l=1}^d(Y_{l,J_l}-Y_{l,J'_l})\right)}, \]
so that
\begin{equation}\label{eq:rhowm}
\rho^{\ast}_{\mathbb{Z}^d}w_{\mathbf{m}}(x,k) = \sum_{\substack{\mathbf{J}=(J_1,\ldots,J_d)\in\mathcal{J}_{\mathbf{r}} \\ \mathbf{J}'=(J'_1,\ldots,J'_d)\in\mathcal{J}_{\mathbf{r}}}} \int_0^1 e^{2\pi i\left((X_{\mathbf{J}}-X_{\mathbf{J}'})\cdot (x-2tk) + t\sum_{l=1}^d(Y_{l,J_l}-Y_{l,J'_l})\right)} \dd t.
\end{equation}
After we perform the integration in $t$, we conclude that the summand depending on $\mathbf{J}$ and $\mathbf{J}'$ is zero unless
\begin{equation}\label{eq:dyadic-relation0}
\sum_{l=1}^d \bigl(Y_{l,J_l}-Y_{l,J'_l} - 2^{M_l+q_l+1}(X_{l,J_l}-X_{l,J'_l})\bigr) = 0.
\end{equation}
We claim that this cannot happen when $\mathbf{J}\neq\mathbf{J}'$. After this is proved, we will know that only diagonal pairs in \eqref{eq:rhowm} survive, i.e.,
$\rho^{\ast}_{\mathbb{Z}^d}w_{\mathbf{m}}(x,k) = \#\mathcal{J}_{\mathbf{r}}$,
which will establish \eqref{eq:xray-vector-exact}.

Fix $\mathbf{J},\mathbf{J}'\in\mathcal{J}_{\mathbf{r}}$ and define
\[ c_{l,j}:=\mathbbm{1}_{J_l}(j)-\mathbbm{1}_{J'_l}(j) \in\{-1,0,1\}, \quad l=1,2,\ldots,d, \quad j=1,2,\ldots,m_l. \]
Rewriting \eqref{eq:dyadic-relation0} with the $c_{l,j}$ constants defined above,
\begin{equation}\label{eq:dyadic-relation}
\sum_{l=1}^d \sum_{\substack{1\leq j\leq m_l\\ j\neq q_l+1}} c_{l,j} \,2^{2M_l} \bigl(2^{2j}-2^{j+q_l+1}\bigr) = 0.
\end{equation}
Since $\#J_l=\#J'_l=r_l$, we have
\begin{equation}\label{eq:dyadic-cardinality}
\sum_{j=1}^{m_l}c_{l,j}=0 \quad\text{for every }l=1,2,\ldots,d.
\end{equation}
We claim that $c_{l,j}=0$ for every $l$ and $j$. The integer exponent of the largest power of $2$ dividing the number $2^{2M_l}(2^{2j}-2^{j+q_l+1})$, i.e., its $2$-adic valuation, equals 
\[ \begin{cases}
2M_l+2j & \text{if } j\leq q_l,\\
2M_l+j+q_l+1 & \text{if } j\ge q_l+2.
\end{cases} \]
Observe that all these numbers are mutually different, over all choices of indices $l$ and $j$ appearing in \eqref{eq:dyadic-relation}. Consequently, if the sum \eqref{eq:dyadic-relation} had at least one nonzero term, then we could divide it by the largest power of $2$ possible. By what we just proved, this would give a zero sum with precisely one odd term, which would be a contradiction.
Hence $c_{l,j}=0$ whenever $j\neq q_l+1$. For each fixed $l$, the only possible remaining nonzero coefficient is thus $c_{l,q_l+1}$, which exists if $q_l+1\leq m_l$, and \eqref{eq:dyadic-cardinality} forces this coefficient to vanish as well.
We conclude $\mathbf{J}=\mathbf{J}'$ and, as we have already said, this completes the proof of \eqref{eq:xray-vector-exact}.

It remains to prove \eqref{eq:numerator-vector-lower}. Expansion of $u_{\mathbf{m}}f_{\mathbf{m}}$ gives
\[ u_{\mathbf{m}}(x,t) f_{\mathbf{m}}(x,t)
= \sum_{\substack{1\leq j_1\leq m_1\\ \cdots\\ 1\leq j_d\leq m_d\\ (I_1,\ldots,I_d)\in\mathcal{J}_{\mathbf{r}}}} 
e^{2\pi i\sum_{l=1}^d\left((2^{M_l+j_l}+X_{l,I_l})x_l+(4^{M_l+j_l}+Y_{l,I_l})t\right)}. \]
For given $(j_1,\ldots,j_d)$ and $(I_1,\ldots,I_d)$ such that $j_l\not\in I_l$ for every $l$, the space-time frequency of the above summand is precisely
\[ \biggl(X_{\mathbf{J}},\sum_{l=1}^dY_{l,J_l}\biggr), \quad J_l:=I_l\cup\{j_l\}\in\mathcal{J}_{l,r_l+1}. \]
Conversely, fix $\mathbf{J}\in\mathcal{J}_{\mathbf{r}+\mathbf{1}}$. For each $l$, choose $j_l\in J_l$ and set $I_l=J_l\setminus\{j_l\}$. This gives exactly $\prod_{l=1}^d (r_l+1)$ representations of the same space-time frequency. Distinct choices of $\mathbf{J}$ give distinct spatial frequencies $X_{\mathbf{J}}$, since each coordinate $X_{l,J_l}=\sum_{j\in J_l}2^{M_l+j}$ determines $J_l$ by uniqueness of binary expansion. Therefore, the Fourier coefficient of $u_{\mathbf{m}}f_{\mathbf{m}}$ at each of these $\prod_{l=1}^d \binom{m_l}{r_l+1}$ frequencies has magnitude at least $\prod_{l=1}^d(r_l+1)$, so orthogonality on $\mathbb{T}^{d+1}$ gives
\[ \|u_{\mathbf{m}}f_{\mathbf{m}}\|_{\mathrm{L}^2(\mathbb{T}^{d+1})}^2 \ge \prod_{l=1}^d (r_l+1)^2\binom{m_l}{r_l+1}. \]
Since $w_{\mathbf{m}}=|f_{\mathbf{m}}|^2$, this is precisely \eqref{eq:numerator-vector-lower}.
\end{proof}

\begin{proof}[Proof of Theorem~\ref{thm:d1-log-intro}]
Choose the solution and the weight respectively as \eqref{eq:logNsolution} and \eqref{eq:logNweight}.
By Lemma~\ref{lm:dyadic-test-estimates}, both the Stein-type and the Mizohata--Takeuchi right-hand sides are equal to
\[ \#A_{\mathbf{m}}\prod_{l=1}^d\binom{m_l}{r_l} = \prod_{l=1}^d m_l\binom{m_l}{r_l}. \]
Also, the left-hand sides can be bounded from below using \eqref{eq:numerator-vector-lower}. Hence, both constants $\mathfrak{S}_{\mathbb{T}^d}(2^{m_{1}+\cdots+m_{d}})$ and $\mathfrak{M}_{\mathbb{T}^d}(2^{m_{1}+\cdots+m_{d}})$ are bounded from below by
\begin{equation}\label{eq:dyadic-quotient-vector}
\prod_{l=1}^d \frac{(r_l+1)^2\binom{m_l}{r_l+1}}{m_l\binom{m_l}{r_l}} = \prod_{l=1}^d \frac{(r_l+1)(m_l-r_l)}{m_l}.
\end{equation}

For the statement on the constants localised to $[-N,N]^d$, take the ``isotropic'' choice $m_1=\cdots=m_d=m\geq2$ and recall that $r_l=\lfloor m/2\rfloor$. Then $A_{\mathbf{m}}\subseteq[0,2^{dm}]^d$ and \eqref{eq:dyadic-quotient-vector} gives
\[ \mathfrak{M}_{\mathbb{T}^d}(2^{dm}),\ \mathfrak{S}_{\mathbb{T}^d}(2^{dm}) \ge \biggl(\frac{(r+1)(m-r)}{m}\biggr)^d \gtrsim \Bigl(\frac{m}{4}\Bigr)^d \gtrsim_d m^d. \]
Choosing $m =\lfloor (\log_2 N)/d \rfloor$ for sufficiently large $N$ proves \eqref{eq:justlog1}.

For the lower bound in \eqref{eq:justlinear}, take the ``anisotropic'' choice $\mathbf{m}=(n,1,\ldots,1)$, so that $\#A_{\mathbf{m}}=n$. In \eqref{eq:dyadic-quotient-vector}, all factors except the first are equal to $1$, while the first factor is $\frac{(r+1)(n-r)}{n}$, where $r=\lfloor n/2\rfloor$.
Thus, for every $n\ge2$,
\[ \widetilde{\mathfrak{M}}_{\mathbb{T}^d}(n),\ \widetilde{\mathfrak{S}}_{\mathbb{T}^d}(n) \ge \frac{(r+1)(n-r)}{n} \gtrsim n. \]
This proves the lower bounds in \eqref{eq:justlinear}. 

Now we turn to the upper bounds. Because of \eqref{eq:per-MT_vs_St}, it only remains to prove the Stein-type upper bound. 
Let
\[ u(x,t) = \sum_{k\in A} a_k e^{2\pi i (x\cdot k + t|k|^2)} \]
be the solution \eqref{eq:per-solution} for a given trigonometric polynomial
\[ u_0(x) = \sum_{k\in A} a_k e^{2\pi i x\cdot k},\quad A\subseteq\mathbb{Z}^d,\ \#A=n,\ a_k\in\mathbb{C}. \]
We estimate pointwise on $\mathbb{T}^{d+1}$ using the Cauchy--Schwarz inequality,
\[ |u(x,t)|^2 \leq \Bigl(\sum_{k\in A}|a_k|\Bigr)^2
\leq n\sum_{k\in A}|a_k|^2
= n\|u_0\|_{\mathrm{L}^2(\mathbb{T}^d)}^2. \]
Thus,
\begin{equation}\label{eq:trivial-card-cs}
\int_{\mathbb{T}} \int_{\mathbb{T}^d} |u(x,t)|^2w(x,t)\dd x\dd t \leq n\|u_0\|_{\mathrm{L}^2(\mathbb{T}^d)}^2 \int_{\mathbb{T}^{d+1}}w
\end{equation}
for every nonnegative weight $w$. On the other hand, for every $k\in\mathbb{Z}^d$, translation-invariance gives
\[ \int_{\mathbb{T}^d}\rho^{\ast}_{\mathbb{Z}^d}w(x,k)\dd x
= \int_{\mathbb{T}^d}\int_{\mathbb{T}} w(x-2tk,t)\dd t\dd x  
= \int_{\mathbb{T}^{d+1}} w. \]
Consequently,
\[ \sup_{x\in\mathbb{T}^d}\rho^{\ast}_{\mathbb{Z}^d}w(x,k) \geq \int_{\mathbb{T}^{d+1}}w \]
for every $k\in\mathbb{Z}^d$. Summing this inequality with weights $|a_k|^2$ over $k\in A$ gives
\begin{equation}\label{eq:trivial-cs2}
\sum_{k\in A}|a_k|^2 \sup_{x\in\mathbb{T}^d}\rho^{\ast}_{\mathbb{Z}^d}w(x,k)
\geq \|u_0\|_{\mathrm{L}^2(\mathbb{T}^d)}^2 \int_{\mathbb{T}^{d+1}}w.
\end{equation}
Combining \eqref{eq:trivial-card-cs} and \eqref{eq:trivial-cs2} we obtain
\[ \widetilde{\mathfrak{M}}_{\mathbb{T}^d}(n) \leq \widetilde{\mathfrak{S}}_{\mathbb{T}^d}(n) \leq n. \qedhere \]
\end{proof}

\subsection{A comment on flat rectangles and triangles}
The one-dimensional construction considered above uses a weight that could not have come from the mere counting of geometric configurations.
Indeed, if $A\subseteq\mathbb{Z}$ is finite, then the one-dimensional resonance equations
\[ k_1+k_3=k_2+k_4, \quad k_1^2+k_3^2=k_2^2+k_4^2 \]
force $\{k_1,k_3\}=\{k_2,k_4\}$, and hence
\[ \|u_A\|_{\mathrm{L}^4(\mathbb{T}^2)}^4=2(\#A)^2-\#A. \]
For $k\in A\subseteq\mathbb{Z}$, define $\triangle^k(A)$ on $\mathbb{Z}$ in the same way as we did on $\mathbb{Z}^2$.
Degenerate isosceles triangles give $\#\triangle^k(A)\ge\#A$ for every $k\in A$. Thus the analogous ``rectangle--triangle'' ratios coming from weights of the form \eqref{eq:2Dweight} remain bounded by $2$:
\[ \frac{\|u_A\|_{\mathrm{L}^4(\mathbb{T}^2)}^4}{\sum_{k\in A}\#\triangle^k(A)}\leq 2,
\quad \frac{\|u_A\|_{\mathrm{L}^4(\mathbb{T}^2)}^4}{\#A\, \max_{k\in A}\limits \#\triangle^k(A)}\leq 2 \]
and thus cannot be used to construct counterexamples to \eqref{eq:per-Stein} and \eqref{eq:per-MT}.

%%%%%%%%%%%%%%%%%%%%%%%%%%%%%%%%%%%%%%%%%%%%%%%%%%%%%%%%%%%%%%

\section*{Declaration of AI usage}
OpenAI's ChatGPT 5.5 Pro found the combinatorial construction of finite lattice sets achieving comparatively many rectangles relative to isosceles triangles. It was prompted on May 23, 2026, with a simple prompt defining rectangles and isosceles triangles in this context and then asking:
\begin{quote}
\emph{Show that there exist finite subsets $A$ of the planar integer lattice such that the ratio (number of rectangles in $A$) / (number of isosceles triangles in $A$) is arbitrarily large.}
\end{quote}
It reasoned for $87$ minutes, and the answer was correct in principle. The construction was fully proofread, clarified and rewritten by the authors. There were several unsuccessful earlier and later attempts at the problem, using both the same tool and other AI tools. 
In later chats, the same tool helped us quantify the lower bounds for the rectangles-to-triangles ratio, which eventually led us to the dimension $d=2$ estimates \eqref{eq:logtologloglog}.
It helped us find the trigonometric weight in one dimension, by adapting the counterexample by Cairo \cite{Cairo} to the discrete setting, which eventually led us to estimates \eqref{eq:justlog1}, \eqref{eq:justlinear} and \eqref{eq:justlog2}.

Google's Gemini 3.1 Pro was used to create the TikZ code for the four figures.

Except for the uses described above, all other ideas, results, proofs, bibliography and writing of the manuscript are entirely the work of the authors.

%%%%%%%%%%%%%%%%%%%%%%%%%%%%%%%%%%%%%%%%%%%%%%%%%%%%%%%%%%%%%

\section*{Acknowledgements and funding}
The first author was supported by EPSRC Grant EP/W032880/1. He thanks Tony Carbery, Susana Guti\'errez, Misha Rudnev and Mari Cruz Vilela for helpful discussion and correspondence.

The second author is supported in part by the Croatian Science Foundation under the project HRZZ-IP-2022-10-5116 (FANAP) and in part by the European Union -- NextGenerationEU through the National Recovery and Resilience Plan 2021--2026, via an institutional grant of the University of Zagreb Faculty of Science, IK IA 1.1.3, Impact4Math.

The third author was supported by JSPS Overseas Research Fellowship and JSPS Kakenhi grant numbers 19K03546, 19H01796 and 21K13806.

The fourth author circulated Question~\ref{quest:combinatorics} from February 2025 to May 2026, and he thanks the large number of colleagues who offered a wide range of perspectives on how to approach it. This question was the last missing part of a programme developed during the project funded by the first author's EPSRC Grant EP/W032880/1, which is now completed by the results of this paper. He was also supported by EPSRC Grant EP/W032880/1 and is currently supported by his EPSRC Fellowship UKRI3285 \textit{New perspectives in phase-space Analysis and Fourier restriction}.

%The second author is supported in part by the Croatian Science Foundation under the project HRZZ-IP-2022-10-5116 (FANAP) and in part by the European Union -- NextGenerationEU through the National Recovery and Resilience Plan 2021--2026, via an institutional grant of the University of Zagreb Faculty of Science, IK IA 1.1.3, Impact4Math.
 %The fourth author was also supported by EPSRC Grant EP/W032880/1 and is currently supported by his EPSRC Fellowship UKRI3285 \textit{New perspectives in phase-space Analysis and Fourier restriction}.

%%%%%%%%%%%%%%%%%%%%%%%%%%%%%%%%%%%%%%%%%%%%%%%%%%%%%%%%%%%%%%%%

\end{document}